\documentclass[12pt]{amsart}
\usepackage[utf8]{inputenc}
\usepackage[margin=1.25in]{geometry}    % page size and margins
\usepackage{amsfonts,amsthm,amsmath,amssymb,mathrsfs}
\usepackage{pgf}                      % for colors
\usepackage[colorlinks=true,citecolor=magenta,linkcolor=blue,urlcolor=blue]{hyperref}
\usepackage{tikz}
\usetikzlibrary{arrows,matrix}
\usepackage[font=scriptsize]{caption}
\usepackage{enumerate, enumitem}
\usepackage[capitalize]{cleveref} 

\newcommand{\NN}{\mathbb{N}}

\DeclareMathOperator{\diag}{diag}

\newcommand{\standout}[1]{\textbf{#1}}    % style for terms defined in text

\tikzset{blueedge/.style={blue,ultra thick},greenedge/.style={green!65!black,ultra thick,densely dotted},pinkedge/.style={magenta,ultra thick,dashed}}

\theoremstyle{plain}
\newtheorem{theorem}{Theorem}[section]
\newtheorem{proposition}[theorem]{Proposition}
\newtheorem{lemma}[theorem]{Lemma}

\theoremstyle{definition}
\newtheorem{example}[theorem]{Example}
\newtheorem{remark}[theorem]{Remark}
\newtheorem{defn}[theorem]{Definition}

\begin{document}

%%%%%%%%%%%%%%%%%%%%%%%%%%%%%%%%%%%%%%%%%%%%%%%%%%%%%%%%%%%%%%%%%%
% METADATA
\title[Metric Dimension of Generalized Hamming Graphs]{Metric Dimension of a Diagonal Family of Generalized Hamming Graphs}
\author{Briana Foster-Greenwood}
\address{Department of Mathematics and Statistics, California State Polytechnic University, Pomona, CA 91768}
\email{brianaf@cpp.edu}
\author{Christine Uhl}
\address{Department of Mathematics, St.\ Bonaventure University,
St\ Bonaventure, NY 14778}
\email{cuhl@sbu.edu}
\date{\today}
% 05C12 Distance in graphs
% 05C15 Coloring of graphs and hypergraphs
% 05C25 Graphs and abstract algebra (groups, rings, fields, etc.)
% 05B30 Other designs, configurations
% 05B05 Combinatorial aspects of block designs
% 05B15 Orthogonal arrays, Latin squares, Room squares
% 05C65 Hypergraphs
% 05C69 Vertex subsets with special properties (dominating sets, independent sets, cliques, etc.)

\begin{abstract}
Classical Hamming graphs are Cartesian products of complete graphs, and two vertices are adjacent if they differ in exactly one coordinate. Motivated by connections to unitary Cayley graphs, we consider a generalization where two vertices are adjacent if they have no coordinate in common. Metric dimension of classical Hamming graphs is known asymptotically, but, even in the case of hypercubes, few exact values have been found. In contrast, we determine the metric dimension for the entire diagonal family of $3$-dimensional generalized Hamming graphs. Our approach is constructive and made possible by first characterizing resolving sets in terms of forbidden subgraphs of an auxiliary edge-colored hypergraph.
\end{abstract}

\maketitle
%%%%%%%%%%%%%%%%%%%%%%%%%%%%%%%%%%%%%%%%%%%%%%%%%%%%%%%%%%%%%%%%%%%%%%%
%%%%%%%%%%%%%%%%%%%%%%%%%%%%%%%%%%%%%%%%%%%%%%%%%%%%%%%%%%%%%%%%%%%%%%%
\section{Introduction}
%%%%%%%%%%%%%%%%%%%%%%%%%%%%%%%%%%%%%%%%%%%
% introduce as game to make it fun, then put into context of Hamming graphs, metric dimension, etc.
Consider an $m\times n$ chessboard with some cells occupied by \standout{landmarks}. A landmark in cell $(i,j)$ \standout{sees} all other cells that are in row $i$ or column $j$. Is it possible to place landmarks on the board so that each unoccupied cell is seen by a different (possibly empty) set of landmarks? What is the minimum number of landmarks required? What if the puzzle is played in higher dimensions on an $n_1\times \cdots\times n_r$ board, where a landmark  sees all other cells that share at least one coordinate with the landmark's cell?

This optimization puzzle is equivalent to finding the metric dimension of a generalized Hamming graph. After providing background on metric dimension and Hamming graphs, we solve the $2$-dimensional puzzle using known results and then devote the rest of the paper to solving $3$-dimensional puzzles on $n\times n\times n$ boards. 

%%%%%%%%%%%%%%%%%%%%%%%%%%%%%%%%%%%%%%%%%%%%%%%%%%%%%%%%%%%%%%%%%
\subsection{Metric dimension}
  Let $G$ be a finite connected graph with vertex set $V$. For vertices $x,y\in V$, define the distance $d(x,y)$ to be the length of the shortest path between $x$ and $y$ in $G$. Given a subset of vertices $W\subseteq V$, whose elements are referred to as \standout{landmarks}, we say $W$ is a \standout{resolving set} (or $W$ \standout{resolves} $G$) provided that for every pair of distinct vertices $x,y\in V-W$, there exists a landmark $w\in W$ such that $d(x,w)\neq d(y,w)$. 
A minimum size resolving set is a \standout{metric basis} for $G$. The \standout{metric dimension} of $G$, denoted $\dim G$, is the size of a metric basis. 

In the graph theory context, metric dimension was introduced independently by Harary and Melter \cite{HararyMelter1976} and Slater \cite{Slater1975} in the 1970s.  Bounds and values for metric dimension and its variants have been found for many graph families. 
See \cite{Tillquist2021} for a nice survey paper on metric dimension and some applications, which include source localization (detecting the source of spread in networks), detecting network motifs, and embedding biological sequence data. See \cite{Kuziak2021} for a survey on the many variants of metric dimension.

%%%%%%%%%%%%%%%%%%%%%%%%%%%%%%%%%%%%%%%%%%%%%%%%%%%%%%%%%%%%%%%
\subsection{Hamming graphs}
While there are many generalizations of Hamming graphs, we adopt the definition from \cite{Sander2010}. For an $r$-tuple of positive integers $n_1,\ldots,n_r$ and a set of distances $K\subseteq\{1,2,\ldots,r\}$, let the \standout{generalized Hamming graph} $HG(n_1,\ldots,n_r;K)$ be the graph with vertex set $V=\{(x_1,\ldots,x_r)\mid 1\leq x_i\leq n_i\}$
 and adjacency defined by $x\sim y$ if and only if there exists $k\in K$ such that $x$ and $y$ differ in exactly $k$ coordinates. 
 
 As noted in \cite{Sander2010}, generalized Hamming graphs include Cartesian products of complete graphs
 \[HG(n_1,\ldots,n_r;1)\cong K_{n_1}\Box\cdots\Box K_{n_r}.\] In particular, this includes the classical Hamming graphs $H(d,q)$ which are the Cartesian products of $d$ copies of the complete graph $K_q$ (with the case of $q=2$ yielding hypercubes, also known as binary Hamming graphs). C\'{a}ceres, et al.~\cite{Caceres2007} determine the metric dimension of a Cartesian product of two complete graphs and bound the metric dimension of the Cartesian product of a complete graph with another graph. Research on resolvability in Hamming graphs includes, for instance, work on coin-weighing problems (e.g., \cite{ErdosRenyi1963}, \cite{Lindstrom1964}, \cite{CantorMills1966},\cite{Lu2022}), asymptotic results for metric dimension of general Cartesian powers \cite{Jiang2019}, and an integer linear programming approach for testing resolvability \cite{Laird2020}. For more details and references, see the survey \cite{Tillquist2021}. We also note that Junnila, et al. \cite{Junnila2019} determine the minimum size of self-locating-dominating codes for the Hamming graphs $H(3,q)$, which is a different problem but has a similar flavor to metric dimension.

For purposes of the puzzle introduced at the beginning of this article, we are interested in the generalized Hamming graphs $HG(n_1,\ldots,n_r;r)$, in which two vertices are adjacent if they have no coordinates in common. For $r\geq 2$ and $n_i\geq 3$, the graphs are connected with diameter two. In particular,
\[ d(x,y) = 
     \begin{cases}
       0 & \text{if $x=y$,}  \\
       1 & \text{if $x$ and $y$ have no coordinates in common,} \\
       2 & \text{if $x$ and $y$ have at least one coordinate in common.}
    \end{cases}\]
(Note that for $r\geq 2$, if $n_1=n_2=2$, the graph $HG(n_1,\ldots,n_r;r)$ is disconnected, and if $n_1=2$ and $n_2,\ldots,n_r\geq 3$, the graph is connected but has diameter $3$.)

In terms of the puzzle, we see that, for $r\geq 2$ and $n_i\geq 3$, the cells of an $n_1\times\cdots\times n_r$ board correspond to vertices in the graph $HG(n_1,\ldots,n_r;r)$, and a landmark $w$ sees vertex $x$ precisely when $d(w,x)=2$. Also, $x$ and $y$ that are not landmarks will be resolved if and only if there is a landmark that sees $x$ or $y$ but not both, i.e., $x$ and $y$ are seen by different sets of landmarks. Therefore a resolving set for the graph $HG(n_1,\ldots,n_r;r)$ corresponds to a solution to the puzzle, and the metric dimension is the minimum number of landmarks required to solve the puzzle.

Note that, alternatively, we could use the graphs $HG(n_1,\ldots,n_r;1,\ldots,r-1)$. These graphs also have diameter two, but the criteria for distance one and distance two are swapped. It follows that the graphs
$HG(n_1,\ldots,n_r;r)$ and $HG(n_1,\ldots,n_r;1,\ldots,r-1)$ have the same resolving sets and metric dimension, so we could use either graph to represent our puzzles.

The generalized Hamming graphs under consideration can be expressed in terms of various graph operations involving Cartesian products, complements, power graphs, and exact distance graphs. Given a graph $G$, we can construct the \standout{graph complement} $\overline{G}$, the \standout{$k$-th power graph} $G^{(k)}$, and the \standout{exact distance-$k$ graph} $G^{[\natural k]}$. All three graphs share the same vertex set as $G$. Distinct vertices are adjacent in the $k$-th power graph if they are at most distance $k$ apart in $G$, while adjacency in the exact distance-$k$ graph requires a distance of exactly $k$. By \cite[Lemma 3.1]{Sander2010}, the graphs $HG(n_1,\ldots,n_r;r)$ and $HG(n_1,\ldots,n_r;1,\ldots,r-1)$ are complements of each other and we have the isomorphisms
\[HG(n_1,\ldots,n_r;r)\cong(K_{n_1}\Box\cdots \Box K_{n_r})^{[\natural r]}\]
and 
\[HG(n_1,\ldots,n_r;1,\ldots,r-1)\cong (K_{n_1}\Box\cdots \Box K_{n_r})^{(r-1)}.\]
For structural results on exact distance graphs of product graphs, see \cite{Bresar2019}.
The above isomorphisms will be useful in solving the $2$-dimensional puzzle.

%%%%%%%%%%%%%%%%%%%%%%%%%%%%%%%%%%%%%%%%%%%%%%%%%%%
\subsection{Two- and three-dimensional puzzles}
By the remarks in the previous subsection, we can solve the $2$-dimensional puzzle on an $m\times n$ board by finding the metric dimension and minimum resolving sets of the graph  $HG(m,n;2)\cong(K_m\Box K_n)^{[\natural 2]}$ or $HG(m,n;1)\cong K_m\Box K_n$. 
Choosing the latter, we can apply \cite[Theorem 6.1]{Caceres2007} that for all $n\geq m\geq 1$,
\[\dim(K_m\square K_n)=\begin{cases}
     \lfloor\frac{2}{3}(n+m-1)\rfloor & \text{if $m\leq n\leq 2m-1$} \\
     n-1 & \text{if $n\geq2m-1$.}
  \end{cases}\]
Moreover, \cite[Lemma 6.2]{Caceres2007} provides necessary and sufficient conditions for a subset of vertices $S$ to be resolving based on relationships between the elements of $S$.

To solve the puzzle on an $n_1\times n_2\times n_3$ board, we want to find the metric dimension of $HG(n_1,n_2,n_3;3)\cong(K_{n_1}\Box K_{n_2}\Box K_{n_3})^{[\natural 3]}$ or $HG(n_1,n_2,n_3;1,2)\cong (K_{n_1}\Box K_{n_2}\Box K_{n_3})^{(2)}$. We are
unaware of any results that give the metric dimension of these graphs, so that will be the focus of the remainder of the paper.
In \cref{sec:lowerbound}, we find a lower bound for the metric dimension.  In \cref{sec:landmarksystems}, we focus our attention on the diagonal family $HG(n,n,n;3)$ and develop a characterization of resolving sets in terms of forbidden edge-colored subgraphs of an auxiliary hypergraph. In \cref{sec:diagonal}, we apply the theorems from \cref{sec:landmarksystems} to construct minimum resolving sets for diagonal Hamming graphs. We end with the comprehensive \cref{thm:ComprehensiveMetDim} for the metric dimension of the Hamming graphs $HG(n,n,n;3)$ for $n \geq 3$ and indicate further research directions with connections to unitary Cayley graphs.  

Throughout the paper, let $\mathcal{N}=\{(n_1,n_2,n_3)\in\NN^3\mid n_1,n_2,n_3\geq 3 \}$ and let the diagonal be $\diag(\mathcal{N}) = \{(n,n,n)\mid n\geq 3\}$. We commonly denote an element of $\mathcal{N}$ as $\mathbf{n}=(n_1,n_2,n_3)$ and increment each coordinate to get $\mathbf{n}+\mathbf{1}=(n_1+1,n_2+1,n_3+1)$. For $n\in\NN$, we let $[n]=\{1,2,\ldots,n\}$.

%%%%%%%%%%%%%%%%%%%%%%%%%%%%%%%%%%%%%%%%%%%%%%%%%%%%%%%%%%%%%%%%%%%%%%%
%%%%%%%%%%%%%%%%%%%%%%%%%%%%%%%%%%%%%%%%%%%%%%%%%%%%%%%%%%%%%%%%%%%%%%%
\section{Lower Bound}\label{sec:lowerbound}
In this section, we prove a lower bound for the metric dimension of the Hamming graphs $HG(n_1,n_2,n_3;3)$.

Given $\mathbf{n}\in\mathcal{N}$ and a subset of vertices $W$ of the Hamming graph $HG(\mathbf{n};3)$, let $W_{i,a}$ be the set of landmarks in the plane $x_i=a$, i.e., 
\[W_{i,a}=\{(x_1,x_2,x_3)\in W\mid x_i=a\}\]
for $1\leq i\leq 3$ and $a\in[n_i]$.
We call the sets $W_{i,a}$ \standout{blocks of color $i$}.

We begin with a constraint on the minimum number of landmarks in the blocks of a given color. Namely,
if $W$ is a resolving set, then the absence of landmarks in a block forces the other blocks of that color to each contain at least three landmarks.  Additionally, if a block has only one landmark in it, then the other blocks of that color are forced to each contain at least two landmarks.  
\begin{lemma}\label{parallel}
If $W$ is a resolving set for the Hamming graph $G=HG(n_1,n_2,n_3;3)$ (each $n_i\geq 3$), then \[|W_{i,a}|+|W_{i,b}|\geq 3\]
for all $1\leq i\leq 3$ and distinct $a,b\in [n_i]$.
\end{lemma}
\begin{proof}
  Let $W$ be a subset of the vertex set of $G$ and let $i=1$ (the proofs for $i=2,3$ are analogous). 
  First suppose $|W_{1,a}|=0$ and $|W_{1,b}|=2$ for some $a,b\in[n_1]$. 
  Say the two landmarks in $W_{1,b}$ are $(b,x_1,y_1)$ and $(b,x_2,y_2)$.
  Now, since $n_2,n_3\geq 3$, we can choose $x_3\in[n_2]-\{x_1,x_2\}$ and $y_3\in[n_3]-\{y_1,y_2\}$ and define $(x,y)$ as follows:
    \[(x,y)=\begin{cases}(x_1,y_2) &\text{if $x_1\neq x_2$ and $y_1\neq y_2$}\\
    (x_1,y_3)&\text{if $x_1=x_2$}\\
    (x_3,y_1)&\text{if $y_1=y_2$.}\end{cases}\]
    To resolve the vertices $(a,x,y)$ and $(b,x,y)$, we would need a landmark that has a coordinate in common with one of the vertices but not the other, i.e., a landmark from $W_{1,a}\cup W_{1,b}$. But $W_{1,a}$ is empty, and, by construction, $(x,y)$ has exactly one coordinate in common with $(x_1,y_1)$ and exactly one coordinate in common with $(x_2,y_2)$, which implies $(a,x,y)$ and $(b,x,y)$ are distance two from the elements of $W_{1,b}$. Hence $(a,x,y)$ and $(b,x,y)$ are unresolved.
    
    Next suppose $|W_{1,a}|=1$ and $|W_{1,b}|=1$. Say the landmarks in $W_{1,a}\cup W_{1,b}$ are $(a,x_1,y_1)$ and $(b,x_2,y_2)$. Defining $(x,y)$ as in the previous case, we can similarly show that $(a,x,y)$ and $(b,x,y)$ are distance two to the landmarks in $W_{1,a}\cup W_{1,b}$ and are hence unresolved.
    
    Finally, if $|W_{1,a}|+|W_{1,b}|<2$, the case is all the worse, for if $W$ was resolving, then we could add another landmark to $W_{1,a}\cup W_{1,b}$ to create a resolving set $W'$ with $|W'_{1,a}|+|W'_{1,b}|=2$, a contradiction.
\end{proof}

We next use the pairwise bound from \cref{parallel} to determine a lower bound on the total number of landmarks required to resolve Hamming graphs $HG(\mathbf{n};3)$.

\begin{theorem}\label{cor:lowerbound}
    For $n_3\geq n_2\geq n_1\geq 3$, the Hamming graph $G=HG(n_1,n_2,n_3;3)$ cannot be resolved with fewer than $2n_3-1$ landmarks.
\end{theorem}

\begin{proof}
    Suppose $W$ resolves $G$. For $1\leq i\leq 3$, the sum $|W_{i,1}|+|W_{i,2}|+\cdots+|W_{i,n_i}|$ is the total number of landmarks in $W$. If the minimum term in this sum is zero, \cref{parallel} says the remaining terms are greater than or equal to three, giving at least $3(n_i-1)$ landmarks. If the minimum term is one, \cref{parallel} says the remaining terms are greater than or equal to two, giving at least $1+2(n_i-1)$ landmarks. If the minimum term is two, we have at least $2n_i$ landmarks. Note that, under the assumption $n_i\geq 3$, we have
    \[2n_i-1<2n_i\leq 3(n_i-1).\]
    So, regardless of the minimum term in the sum, we have $2n_i-1$ as a lower bound on the number of landmarks in $W$. In particular, $2n_3-1$ is a lower bound (and with the assumption $n_3\geq n_2\geq n_1$, it is the greatest of the lower bounds $2n_i-1$).
\end{proof}

In \cref{sec:diagonal}, we will show the bound in \cref{cor:lowerbound} is sharp by resolving the Hamming graphs $HG(n,n,n;3)$ for $n\geq 5$ in $2n-1$ landmarks.
In preparation, we develop an equivalent characterization of resolving sets.

%%%%%%%%%%%%%%%%%%%%%%%%%%%%%%%%%%%%%%%%%%%%%%%%%%%%%%%%%%%%%%%%%%%%%
%%%%%%%%%%%%%%%%%%%%%%%%%%%%%%%%%%%%%%%%%%%%%%%%%%%%%%%%%%%%%%%%%%%%%
\section{Landmark Graphs and Forbidden Configurations}\label{sec:landmarksystems}

A priori, determining whether a subset of vertices $W$ is a resolving set for a Hamming graph $HG(\mathbf{n};3)$ involves checking whether all pairs of vertices are resolved. Instead of doing this directly, we shift our focus to relationships between the elements of $W$. We define an edge-colored hypergraph based on $W$ and then characterize whether $W$ is resolving in terms of forbidden edge-colored subgraphs. This characterization will be used in \cref{sec:diagonal} to construct minimum resolving sets for the diagonal family $HG(n,n,n;3)$. 

\subsection{Landmark graphs and systems}
As in \cref{sec:lowerbound}, given a subset of vertices $W$ in a Hamming graph $HG(\mathbf{n};3)$, let $W_{i,a}$ be the set of elements in $W$ whose $i$-th coordinate is $a$.  
To highlight relationships between the coordinates of elements of $W$, we define the \standout{landmark graph $\mathcal{G}(W)$} to be the hypergraph with vertex set $W$ and hyperedges those $W_{i,a}$ that are nonempty. We often refer to $\mathcal{G}(W)$ as the graph of $W$. The graph of $W$ is $3$-regular and has a proper $3$-edge-coloring obtained by assigning color $i$ to the sets $W_{i,a}$. Note that the hyperedges of color $i$ form a partition of $W$. In practice, we will think of the first color as blue, second color as green, and third color as pink. As an example, \cref{fig:landmarkgraph} shows the landmark graph of a resolving set for $HG(3,3,3;3)$.

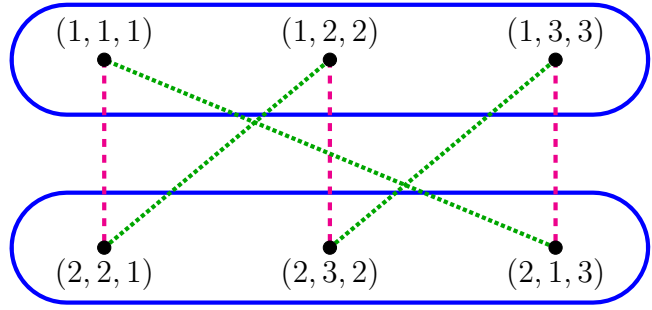
\begin{figure}
    \centering
\begin{tikzpicture}[scale=1]
    \draw[blueedge,double distance=40pt,line cap=round] (-.5,2.5)--(6.5,2.5);
    \draw[blueedge,double distance=40pt,line cap=round] (-.5,0)--(6.5,0);
    \draw[pinkedge] (0,2.5)--(0,0);
    \draw[pinkedge] (3,2.5)--(3,0);
    \draw[pinkedge] (6,2.5)--(6,0);
    \draw[greenedge] (0,0)--(3,2.5);
    \draw[greenedge] (3,0)--(6,2.5);
    \draw[greenedge] (6,0)--(0,2.5);
    \draw[fill=black] (0,2.5) circle (2.5pt) node[above] {$(1,1,1)$};
    \draw[fill=black] (3,2.5) circle (2.5pt) node[above] {$(1,2,2)$};
    \draw[fill=black] (6,2.5) circle (2.5pt) node[above] {$(1,3,3)$};
    \draw[fill=black] (0,0) circle (2.5pt) node[below] {$(2,2,1)$};
    \draw[fill=black] (3,0) circle (2.5pt) node[below] {$(2,3,2)$};
    \draw[fill=black] (6,0) circle (2.5pt) node[below] {$(2,1,3)$};
\end{tikzpicture}
    \caption{Landmark graph $\mathcal{G}(W)$ of a resolving set for $HG(3,3,3;3)$. The six landmarks in $W$ are $(1,1,1)$, $(1,2,2)$, $(1,3,3)$, $(2,2,1)$, $(2,3,2)$, and $(2,1,3)$. Landmarks in the same blue (solid) hyperedge have the same first coordinate; landmarks connected by a green (dotted) edge have the same second coordinate; and landmarks connected by a pink (dashed) edge have the same third coordinate.}
    \label{fig:landmarkgraph}
\end{figure}

The following lemma tells us a necessary condition on the structure of the landmark graph if a diagonal Hamming graph is to be resolved with the lower bound of $2n-1$ landmarks. We use the phrase ``plain edge'' to emphasize an edge with two endpoints (as opposed to a loop or larger hyperedge).

\begin{lemma}\label{howmany}
   Let $n\geq3$. If $W$ is a resolving set of $HG(n,n,n;3)$ of size $2n-1$,
   then the landmark graph $\mathcal{G}(W)$ has exactly one loop of each color and exactly $(n-1)$ plain edges of each color.
\end{lemma}
\begin{proof}
  Suppose $W$ is a resolving set of $HG(n,n,n;3)$ of size $2n-1$.
  First, note that if the $n$ blocks of color $i$ all have size $2$ or greater, then $|W|\geq 2n$, a contradiction. Hence there must be at least one block of size $0$ or $1$. But if there is a block of color $i$ that has size $0$, then, by \cref{parallel}, the other blocks of color $i$ all have size $3$ or greater, which implies $|W|\geq 3(n-1)> 2n-1$, also a contradiction. Thus there must be a block $W_{i,a}$ of size $1$. Then, by \cref{parallel}, all other blocks $W_{i,b}$ with $b\neq a$ have size at least $2$. Moreover, to get exactly $2n-1$ elements in $W$, the sizes of $W_{i,b}$ for $b\neq a$ cannot exceed $2$.
\end{proof}

To optimally resolve the graphs in the diagonal family, we are interested in subsets $W$ that will produce a graph $\mathcal{G}(W)$ with the structure given in \cref{howmany}. Our strategy is to start with graphs that produce an equal number of plain edges of each color and then add on loops. Though not necessary, we attempt to put all loops at the same vertex to make the analysis and constructions easier. This leads us to the following definition.

\begin{defn}
Let $n\geq 3$ and $\mathbf{n}=(n,n,n)$. For $W$ a subset of vertices of the graph $HG(\mathbf{n};3)$, we say $W$ is a \standout{2-basic landmark system} provided that:
\begin{enumerate}
    \item $|W_{i,a}|=2$ for all $1\leq i\leq 3$ and $a\in[n]$; and
    \item $|W_{i,a}\cap W_{j,b}|\leq 1$ for all $1\leq i < j\leq 3$ and $a,b\in[n]$.
\end{enumerate}
Letting $u=(n+1,n+1,n+1)$, we can extend a $2$-basic landmark system for $HG(\mathbf{n};3)$ to a \standout{triple-looped landmark system} $W\cup\{u\}$ for $HG(\mathbf{n}+\mathbf{1};3)$. 
\end{defn}

Condition (1) ensures that $\mathcal{G}(W)$ only has edges with two endpoints, and condition (2) ensures that $\mathcal{G}(W)$ has no double edges. Thus, when $W$ is a $2$-basic landmark system, $\mathcal{G}(W)$ is a simple graph with $n$ edges of each color. Note that the definition of a landmark system makes no assertion as to whether $W$ is a resolving set. Our goal is to find further conditions, based on subgraphs of $\mathcal{G}(W)$, which will allow us to determine if $W$ is a resolving set.

%%%%%%%%%%%%%%%%%%%%%%%%%%%%%%%%%%%%%%%%%%%%%%%%%%%%%%%%%%%%%%%%%%%%
\subsection{Subgraphs}
Before continuing, we establish notation and definitions for subgraphs.
We denote a path with $n$ vertices as $P_n$, a cycle with $n$ vertices as $C_n$, and a complete bipartite graph with parts of size $m$ and $n$ as $K_{m,n}$. Some of our graphs will have loops and multiple edges, so we denote a vertex with $n$ loops as $L_n$ and a double edge between the same two vertices as $D_2$. We write $G\cup H$ for the disjoint union of graphs $G$ and $H$.

We say edge $e$ \standout{covers} vertex $u$ if $u\in e$. A subset of edges $E'$ determines an \standout{edge-induced subgraph} consisting of the edges in $E'$ and the vertices they cover.
In a graph with an edge-coloring, a \standout{rainbow subgraph} is a subgraph whose edges are all different colors. 
We can identify (almost) every vertex in a Hamming graph with an edge-induced rainbow subgraph of a landmark graph.

\begin{defn}
Let $\mathbf{n}\in\mathcal{N}$ and let $W$ be a subset of vertices of the Hamming graph $HG(\mathbf{n};3)$. 
Given a vertex $\alpha = (a_1,a_2,a_3)$ of the Hamming graph, define the \standout{footprint of $\alpha$} (relative to $W$) to be the subgraph of $\mathcal{G}(W)$ induced by the edges $W_{i,a_i}$. 
If $W_{i,a_i} = \varnothing$ for all $1 \leq i \leq 3$, then there is no edge-induced subgraph and $\alpha$ has no footprint.
\end{defn}

\begin{figure}
    \centering
  \hfill\begin{tikzpicture} % triangle
      \draw[blueedge] ({330+120}:.75)--({330+240}:.75);
    \draw[greenedge] ({330+240}:.75)--({330}:.75);
            \draw[pinkedge] ({330+120}:.75)--({330}:.75);
\foreach \x in {1,2,3}
\draw[fill=white] ({330+120*\x}:.75) circle (2pt);
    \end{tikzpicture}\hfill
    \begin{tikzpicture} % path P4
    \draw[blueedge] (0,0)--(0,1);
    \draw[pinkedge] (0,1)--(1,1);
    \draw[greenedge] (0,0)--(1,0);
       \draw[fill=white] (0,0) circle (2pt);
       \draw[fill=white] (1,0) circle (2pt);
       \draw[fill=white] (1,1) circle (2pt);
       \draw[fill=white] (0,1) circle (2pt);
    \end{tikzpicture}\hfill
    \begin{tikzpicture} % P2 U P3
      \draw[blueedge] ({18+72}:.75)--({18+144}:.75);
    \draw[greenedge] ({18+72}:.75)--({18+0}:.75);
            \draw[pinkedge] ({18+3*72}:.75)--({18+4*72}:.75);
\foreach \x in {1,2,3,4,5}
\draw[fill=white] ({18+72*\x}:.75) circle (2pt);
    \end{tikzpicture}\hfill
    \begin{tikzpicture} % 3P2
        \draw[blueedge] ({120}:.75)--({180}:.75);
        \draw[greenedge] ({240}:.75)--({300}:.75);
        \draw[pinkedge] ({0}:.75)--({60}:.75);
        \foreach \x in {1,2,3,4,5,6}
            \draw[fill=white] ({60*\x}:.75) circle (2pt);
    \end{tikzpicture}\hfill\text{}\vskip36pt\hfill
    \begin{tikzpicture} % L2 U P2
        \draw[blueedge] (0,0) to[out=135,in=180] (0,0.5);
        \draw[blueedge] (0,0) to[out=45,in=0] (0,0.5);
        \draw[greenedge] (0,0) to[out=180,in=180] (0,0.75);
        \draw[greenedge] (0,0) to[out=0,in=0] (0,0.75);
        \draw[pinkedge] (1,0)--(1,1);
        \draw[fill=white] (0,0) circle (2pt);
        \draw[fill=white] (1,0) circle (2pt);
        \draw[fill=white] (1,1) circle (2pt);
    \end{tikzpicture}\hfill
    \begin{tikzpicture} % P3 U L1
        \draw[blueedge] (0,0)--(0,1);
        \draw[pinkedge] (0,1)--(1,1);
        \draw[greenedge] (1,0) to[out=135,in=180] (1,.5);
        \draw[greenedge] (1,0) to[out=45,in=0] (1,.5);
        \draw[fill=white] (0,0) circle (2pt);
        \draw[fill=white] (1,0) circle (2pt);
        \draw[fill=white] (1,1) circle (2pt);
        \draw[fill=white] (0,1) circle (2pt);
    \end{tikzpicture}\hfill
    \begin{tikzpicture} % 2P2 U L1
      \draw[blueedge] ({18+72}:.75)--({18+144}:.75);
      \draw[greenedge] ({18}:.75) to[out=135,in=180] ++(0,0.5);
      \draw[greenedge] ({18}:.75) to[out=45,in=0] ++(0,.5);
      \draw[pinkedge] ({18+3*72}:.75)--({18+4*72}:.75);
      \foreach \x in {1,2,3,4,5}
          \draw[fill=white] ({18+72*\x}:.75) circle (2pt);
    \end{tikzpicture}\hfill\text{}
 \caption{If $W$ is a $2$-basic landmark system, then there are four possible footprints for a non-landmark: rainbow $C_3$, $P_4$, $P_3\cup P_2$, and $3P_2$. If $W$ is a triple-looped landmark system, then there are three additional possible footprints: rainbow $L_2\cup P_2$, $P_3\cup L_1$, and $2P_2\cup L_1$. Different line styles correspond to different edge colors.}
    \label{fig:dinosaurtracks}
\end{figure}
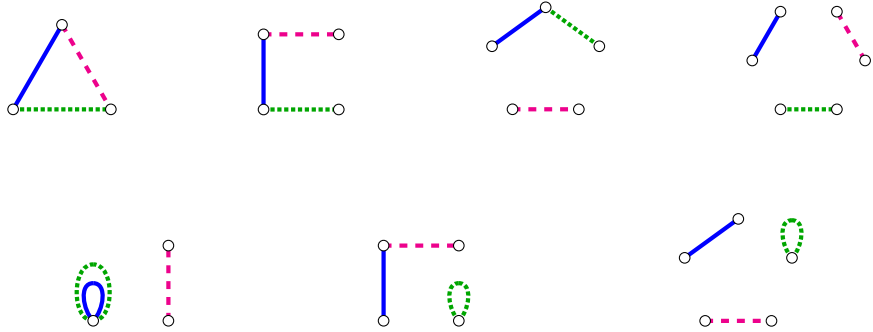

In \cref{fig:dinosaurtracks}, we show the possible footprints of a non-landmark vertex in a Hamming graph $HG(\mathbf{n};3)$. There are four possible footprints assuming that $W$ is a $2$-basic landmark system, and three additional footprints assuming $W$ is a triple-looped landmark system. The vertices of the footprint of $\alpha$ are the landmarks that have a coordinate in common with $\alpha$,  i.e., the vertices in $W$ that have distance $2$ from $\alpha$ in $HG(\mathbf{n};3)$. Notice that the footprint of a landmark has three edges/loops that share a common vertex, which in the case of a triple-looped landmark system can either be a $K_{1,3}$ or an $L_3$. We omit these from \cref{fig:dinosaurtracks} since landmarks are automatically resolved.

Because we are working with Hamming graphs of diameter two, the following straightforward but important observation allows us to determine whether $W$ is resolving by examining how footprints overlap.

\begin{remark}
  Let $W$ be a subset of vertices of a Hamming graph $HG(\mathbf{n};3)$.
  Note that in the Hamming graph, a vertex $(a_1,a_2,a_3)\notin W$ has distance two to the landmarks in $W_{1,a_1}\cup W_{2,a_2}\cup W_{3,a_3}$ and distance one to the remaining landmarks in $W$. 
  It follows that distinct vertices $\alpha=(a_1,a_2,a_3)$ and $\beta=(b_1,b_2,b_3)$ not in $W$ are resolved if and only if 
  \[W_{1,a_1}\cup W_{2,a_2}\cup W_{3,a_3}\neq W_{1,b_1}\cup W_{2,b_2}\cup W_{3,b_3},\]
  i.e., the footprints of $\alpha$ and $\beta$ cover different sets of vertices.
\end{remark}

%%%%%%%%%%%%%%%%%%%%%%%%%%%%%%%%%%%%%%%%%%%%%%%%%%%%%%%%%%%%%%%%%%%%
\subsection{Forbidden configurations for two-basic landmark systems}

We are now ready to characterize whether $W$ is resolving based on properties of the landmark graph of $W$. In particular, we will find necessary and sufficient conditions in terms of forbidden edge-colored subgraphs.  We begin with two examples of configurations that cannot occur in the landmark graph of a resolving set.  

\begin{example}[Forbidden $4$-cycle]\label{ex:fourcycle}
  Let $\mathbf{n}\in\mathcal{N}$ and let $W$ be a subset of vertices of the graph $HG(\mathbf{n};3)$.
  Suppose the graph of $W$ contains a $4$-cycle, as illustrated in \cref{fig:verboten}, with opposite blue edges $W_{1,a_1}=\{w_1,w_2\}$ and $W_{1,b_1}=\{w_3,w_4\}$, a green edge $W_{2,a_2}=\{w_2,w_3\}$, and a pink edge $W_{3,a_3}=\{w_1,w_4\}$. Then the vertices $(a_1,a_2,a_3)$ and $(b_1,a_2,a_3)$ are unresolved in $HG(\mathbf{n};3)$ since neither is a landmark and \[W_{1,a_1}\cup W_{2,a_2}\cup W_{3,a_3}=\{w_1,w_2,w_3,w_4\}=W_{1,b_1}\cup W_{2,a_2}\cup W_{3,a_3}.\]
  Similarly, if the edge colors are permuted, there will be an unresolved pair of vertices.
\end{example}

\begin{example}[Forbidden $6$-cycle]\label{ex:plainhex}
  Let $\mathbf{n}\in\mathcal{N}$ and let $W$ be a subset of vertices of the graph $HG(\mathbf{n};3)$.
  Suppose the graph of $W$ contains a $6$-cycle, as illustrated in \cref{fig:verboten}, with opposite blue edges $W_{1,a_1}=\{w_1,w_2\}$ and $W_{1,b_1}=\{w_4,w_5\}$; opposite green edges  $W_{2,a_2}=\{w_5,w_6\}$ and $W_{2,b_2}=\{w_2,w_3\}$; and opposite pink edges $W_{3,a_3}=\{w_3,w_4\}$ and $W_{3,b_3}=\{w_1,w_6\}$. Then the vertices $(a_1,a_2,a_3)$ and $(b_1,b_2,b_3)$ are unresolved in $HG(\mathbf{n};3)$ since neither is a landmark and
  \[W_{1,a_1}\cup W_{2,a_2}\cup W_{3,a_3}=\{w_1,w_2,w_3,w_4,w_5,w_6\}=W_{1,b_1}\cup W_{2,b_2}\cup W_{b_3}.\]
  \end{example}

\begin{figure}
    \centering
    \hfill
    \begin{tikzpicture}[scale=1.75] % forbidden 4-cycle
        \draw[blueedge] (0,0)--(0,1);
        \draw[blueedge] (1,0)--(1,1);
        \draw[pinkedge] (0,1)--(1,1);
        \draw[greenedge] (0,0)--(1,0);
        \draw[fill=white] (0,0) circle (2pt);
        \draw[fill=white] (0,1) circle (2pt);
        \draw[fill=white] (1,1) circle (2pt);
        \draw[fill=white] (1,0) circle (2pt);
    \end{tikzpicture}\hfill
    \begin{tikzpicture}[scale=1.75] % forbidden 6-cycle
        \draw[blueedge] (60:.75)--(120:.75);
        \draw[blueedge] (240:.75)--(300:.75);
        \draw[greenedge] (0:.75)--(60:.75);
        \draw[greenedge] (180:.75)--(240:.75);
        \draw[pinkedge] (300:.75)--(360:.75);
        \draw[pinkedge] (120:.75)--(180:.75);
        \foreach \x in {1,...,6}
            \draw[fill=white] ({60*\x}:.75) circle (2pt);
    \end{tikzpicture}\hfill
    \begin{tikzpicture}[scale=1.75] % rainbow triangle
        \draw[blueedge] ({330+120}:.75)--({330+240}:.75);
        \draw[greenedge] ({330+240}:.75)--({330}:.75);
        \draw[pinkedge] ({330+120}:.75)--({330}:.75);
        \foreach \x in {1,2,3}
            \draw[fill=white] ({330+120*\x}:.75) circle (2pt);
    \end{tikzpicture}\hfill\text{}
    \caption{The landmark graph of a resolving set must avoid forbidden $4$-cycles and forbidden $6$-cycles. The landmark graph of a resolving triple-looped landmark system must also avoid rainbow triangles. Solid edges are blue, dashed edges are pink, and dotted edges are green.}
    \label{fig:verboten}
\end{figure}
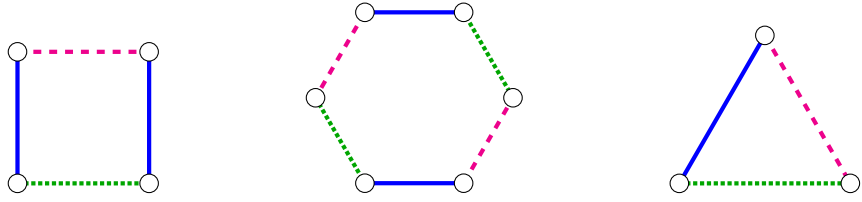

When either the $4$-cycle from \cref{ex:fourcycle} (possibly with colors permuted) or the $6$-cycle from \cref{ex:plainhex} appear in a landmark graph, the Hamming graph $HG(\mathbf{n};3)$ is unresolved.
We next show that when $W$ is a $2$-basic landmark system, these are in fact the only configurations we must avoid to guarantee $W$ is a resolving set.

\begin{theorem}\label{lem:basicverboten}
Let $\mathbf{n}\in\diag(\mathcal{N})$. If $W$ is a $2$-basic landmark system for the Hamming graph $HG(\mathbf{n};3)$, then $W$ is a resolving set if and only if the landmark graph $\mathcal{G}(W)$ avoids the following forbidden subgraphs:
\begin{enumerate}
    \item a $4$-cycle that contains all three colors of edges
    \item a $6$-cycle with opposite edges of the same color.
\end{enumerate}
\end{theorem}

\begin{proof}
    Using \cref{ex:fourcycle} and \cref{ex:plainhex}, we can see that $W$ is not resolving when the landmark graph $\mathcal{G}(W)$ contains a $4$-cycle that contains all three colors of edges or a $6$-cycle with opposite edges of the same color.
    
    Conversely, suppose $W$ is a $2$-basic landmark system whose graph $\mathcal{G}(W)$ avoids forbidden $4$-cycles and $6$-cycles. We will show that $W$ is a resolving set. Suppose $\alpha=(a_1, a_2, a_3)$ and $\beta=(b_1, b_2, b_3)$ are distinct vertices in $HG(\mathbf{n};3)$ and that neither is a landmark. We will show the edges $W_{i,a_i}$ cannot cover the same set of vertices as the edges $W_{i,b_i}$.
    We consider cases based on the number of vertices covered. Throughout, we use the fact that $\mathcal{G}(W)$ is a simple graph with a proper $3$-edge coloring.
    
    Case 1 ($3$ vertices). If the edges $W_{i,a_i}$ and the edges $W_{i,b_i}$ cover the same set of three vertices, then we must have $W_{i,a_i}=W_{i,b_i}$ for $i=1,2,3$ since it is not possible to have two different edges of the same color in a subgraph with only three vertices. But this contradicts our assumption that $\alpha$ and $\beta$ are distinct.
    
    Case 2 ($4$ vertices). Suppose the edges $W_{i,a_i}$ and the edges $W_{i,b_i}$ cover the same set of four vertices. Since $\alpha$ is not a landmark, the edges $W_{i,a_i}$ cannot induce a $K_{1,3}$ and hence must induce a rainbow path $P_4$. Likewise, the edges $W_{i,b_i}$ induce a rainbow path $P_4$. The only way for two different rainbow paths to cover all four vertices is to make a $4$-cycle with one set of opposite edges the same color and the other two edges different colors, but we are assuming that $\mathcal{G}(W)$ avoids this forbidden configuration.
    
    Case 3 ($5$ vertices). Suppose the edges $W_{i,a_i}$ and the edges $W_{i,b_i}$ cover the same set of five vertices. Then each set of edges induces a rainbow $P_2\cup P_3$. Up to permuting colors, suppose the first $P_2\cup P_3$ has blue edge $W_{1,a_1}=\{x,y\}$, green edge $W_{2,a_2}=\{u,v\}$, and pink edge $W_{3,a_3}=\{u,w\}$, as in \cref{fig:p2p3case}[A]. If the second rainbow $P_2\cup P_3$ also has blue edge $\{x,y\}$, then since there are no double edges, the green and pink edges must also coincide, i.e., $W_{2,a_2}=W_{2,b_2}$ and $W_{3,a_3}=W_{3,b_3}$, contradicting $\alpha$ and $\beta$ being distinct. The only other option is for the second rainbow $P_2\cup P_3$ to have blue edge $W_{1,b_1}=\{v,w\}$ as in \cref{fig:p2p3case}[B].
    If the blue edge $vw$ is the $P_2$, then the edges $W_{2,b_2}$ and $W_{3,b_3}$ would have to create a green-pink path $xuy$ or $yux$, but this is not possible with a proper edge coloring. If the blue edge $vw$ is part of the $P_3$, then the $P_2$ would have to create a double edge between $x$ and $y$, which is not allowed.
    
    Case 4 ($6$ vertices). Suppose the edges $W_{i,a_i}$ and the edges $W_{i,b_i}$ cover the same set of six vertices. Then each set of edges creates a rainbow matching $3P_2$. We will show the only way this can happen is if the edges of the matchings together form a forbidden $6$-cycle. Suppose the first rainbow matching has blue edge $W_{1,a_1}=\{u,v\}$, green edge $W_{2,a_2}=\{w,x\}$, and pink edge $W_{3,a_3}=\{y,z\}$, as in \cref{fig:p2p3case}[C and D]. The two matchings must differ by at least one edge, so up to permuting colors, suppose they have different pink edges. Then $W_{3,b_3}$ must connect a vertex from the blue edge $W_{1,a_1}$ to a vertex from the green edge $W_{2,a_2}$. Without loss of generality, say $W_{3,b_3}=\{v,w\}$. Now, the blue edge $W_{1,b_1}$ and green edge $W_{2,b_2}$ must cover the vertices $u$, $x$, $y$, and $z$. There are two ways to do this, leading to either the $6$-cycle $uvwxyzu$ or $uvwxzyu$. Either way, the pattern of edge colors around the cycle is blue, pink, green, blue, pink, green. But this is a contradiction since we assumed that the graph of $W$ has no forbidden $6$-cycles.
    
    Thus, in all cases, the edges $W_{i,a_i}$ and the edges $W_{i,b_i}$ cannot cover the same set of vertices and so $\alpha$ and $\beta$ are resolved. Hence $W$ is a resolving set.
\end{proof}

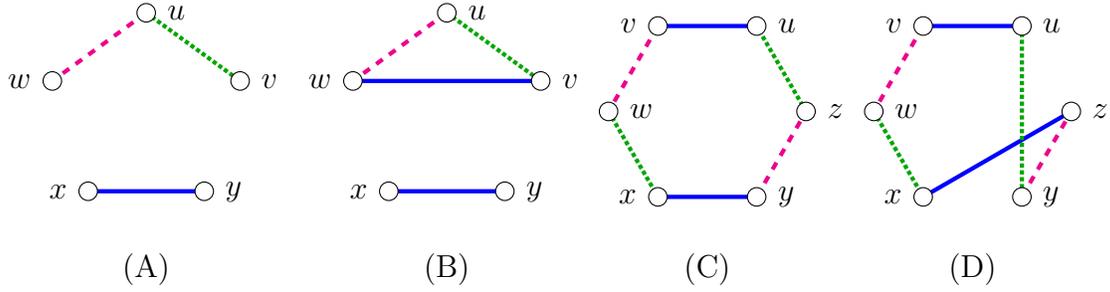
\begin{figure}
    \centering
    \begin{tikzpicture}[scale=1.75] % Case 3(A)
        \draw[pinkedge] ({18+72}:.75)--({18+144}:.75);
        \draw[greenedge] ({18+72}:.75)--({18+0}:.75);
        \draw[blueedge] ({18+3*72}:.75)--({18+4*72}:.75);
        \foreach \x in {1,2,3,4,5}
           \draw[fill=white] ({18+72*\x}:.75) circle (2pt);
        \draw ({18+72}:.75) node[right] {~$u$};
        \draw ({18+144}:.75) node[left] {$w$\,\,};
        \draw ({18+0}:.75) node[right] {~$v$};
        \draw ({18+3*72}:.75) node[left] {$x$\,\,};
        \draw ({18+4*72}:.75) node[right] {~$y$};
        \draw ({270}:1.2) node{(A)};
    \end{tikzpicture}
    \begin{tikzpicture}[scale=1.75] % Case 3(B)
        \draw[pinkedge] ({18+72}:.75)--({18+144}:.75);
        \draw[greenedge] ({18+72}:.75)--({18+0}:.75);
        \draw[blueedge] ({18+3*72}:.75)--({18+4*72}:.75);
    % \draw[magenta,ultra thick,dashed] ({18+3*72}:.75)--({18+0}:.75);
    % \draw[magenta,ultra thick,dashed] ({18+4*72}:.75)--({18+0}:.75);
        \draw[blueedge] ({18+144}:.75)--({18+0}:.75);
    % \draw[green!75!black,ultra thick,dashed] ({18+3*72}:.75)--({18+144}:.75);
    % \draw[green!75!black,ultra thick,dashed] ({18+4*72}:.75)--({18+144}:.75);
        \foreach \x in {1,2,3,4,5}
            \draw[fill=white] ({18+72*\x}:.75) circle (2pt);
        \draw ({18+72}:.75) node[right] {~$u$};
        \draw ({18+144}:.75) node[left] {$w$\,\,};
        \draw ({18+0}:.75) node[right] {~$v$};
        \draw ({18+3*72}:.75) node[left] {$x$\,\,};
        \draw ({18+4*72}:.75) node[right] {~$y$};
        \draw ({270}:1.2) node{(B)};
    \end{tikzpicture}
    \begin{tikzpicture}[scale=1.75] % Case 4(C)
        \draw[blueedge] (60:.75)--(120:.75);
        \draw[blueedge] (240:.75)--(300:.75);
        \draw[greenedge] (0:.75)--(60:.75);
        \draw[greenedge] (180:.75)--(240:.75);
        \draw[pinkedge] (300:.75)--(360:.75);
        \draw[pinkedge] (120:.75)--(180:.75);
        \foreach \x in {1,...,6}
            \draw[fill=white] ({60*\x}:.75) circle (2pt);
        \draw ({60}:.75) node[right] {~$u$};
        \draw ({120}:.75) node[left] {$v$\,\,};
        \draw ({180}:.75) node[right] {~$w$};
        \draw ({240}:.75) node[left] {$x$\,\,};
        \draw ({300}:.75) node[right] {~$y$};
        \draw ({360}:.75) node[right] {~$z$};
        \draw ({270}:1.2) node{(C)};
    \end{tikzpicture}
    \begin{tikzpicture}[scale=1.75] % Case 4(D)
        \draw[blueedge] (60:.75)--(120:.75);
        \draw[greenedge] (180:.75)--(240:.75);
        \draw[pinkedge] (300:.75)--(360:.75);
        \draw[blueedge] (240:.75)--(360:.75);
        \draw[greenedge] (60:.75)--(300:.75);
        \draw[pinkedge] (120:.75)--(180:.75);
        \foreach \x in {1,...,6}
            \draw[fill=white] ({60*\x}:.75) circle (2pt);
        \draw ({60}:.75) node[right] {~$u$};
        \draw ({120}:.75) node[left] {$v$\,\,};
        \draw ({180}:.75) node[right] {~$w$};
        \draw ({240}:.75) node[left] {$x$\,\,};
        \draw ({300}:.75) node[right] {~$y$};
        \draw ({360}:.75) node[right] {~$z$};
        \draw ({270}:1.2) node{(D)};
    \end{tikzpicture}\hfill\text{}
    \caption{Illustrations for Cases $3$ and $4$ in the proof of \cref{lem:basicverboten}. Line styles correspond to the edge colors blue (solid), green (dotted), and pink (dashed). (A) and (B) It is not possible for two different rainbow $P_2\cup P_3$ graphs to cover the same set of five vertices. (C) and (D) If two different rainbow matchings cover the same set of six vertices, the edges together form a forbidden $6$-cycle.}
    \label{fig:p2p3case}
\end{figure}

%%%%%%%%%%%%%%%%%%%%%%%%%%%%%%%%%%%%%%%%%%%%%%%%%%%%%%%%%%%%%%%%%%%%%%%
\subsection{Forbidden configurations for triple-looped landmark systems}
We next investigate conditions for when a triple-looped landmark system will resolve a Hamming graph.
We will use these conditions in \cref{sec:diagonal} to construct resolving sets for most of the Hamming graphs $HG(n,n,n;3)$ in $2n-1$ landmarks.

Recall that the graph of a triple-looped landmark system is obtained from the graph of a $2$-basic landmark system by adding a vertex with a triple loop.
The next example shows that if there is a rainbow triangle in the graph of a $2$-basic landmark system, then the extension to a triple-looped landmark system will not be a resolving set.
  
\begin{example}[Triangle union triple loop]\label{ex:tripleloop}
  Let $\mathbf{n}\in\mathcal{N}$ and let $W$ be a subset of vertices of the graph $HG(\mathbf{n};3)$.
  Suppose the graph of $W$ contains a rainbow triangle with blue edge $W_{1,a_1}=\{w_1,w_2\}$, green edge $W_{2,a_2}=\{w_2,w_3\}$, and pink edge $W_{3,a_3}=\{w_3,w_1\}$, as in \cref{fig:verboten}. Letting $u=(n_1+1,n_2+1,n_3+1)$, the extension of $W$ to a triple-looped landmark system $W\cup\{u\}$ does not resolve $HG(\mathbf{n}+\mathbf{1};3)$; for instance, the vertices $(a_1,a_2,n_3+1)$ and $(a_1,n_2+1,a_3)$ are unresolved since neither is a landmark and
  \[W_{1,a_1}\cup W_{2,a_2}\cup W_{3,n_3+1}=\{w_1,w_2,w_3,u\}=W_{1,a_1}\cup W_{2,n_2+1}\cup W_{3,a_3}.\]
\end{example}

The next theorem is analogous to \cref{lem:basicverboten}, but allows for a triple loop at a single vertex.  Furthermore this theorem allows us to construct minimum resolving sets in the next section.

\begin{theorem}\label{lem:verboten}
Let $\mathbf{n}\in\diag(\mathcal{N})$ and suppose $W$ is a $2$-basic landmark system for $HG(\mathbf{n};3)$. Let $u=(n+1,n+1,n+1)$. Then the triple-looped landmark system $W\cup\{u\}$ resolves $HG(\mathbf{n}+\mathbf{1};3)$ if and only if the landmark graph $\mathcal{G}(W)$ avoids the following forbidden subgraphs:
\begin{enumerate}
    \item a $4$-cycle that contains all three colors of edges,
    \item a $6$-cycle with opposite edges of the same color, and
    \item a rainbow triangle.
\end{enumerate}
\end{theorem}

\begin{proof}
    Let $V_n$ and $V_{n+1}$ be the vertex sets of the Hamming graphs $HG(\mathbf{n};3)$ and $HG(\mathbf{n}+\mathbf{1};3)$, respectively.
    If a forbidden $4$-cycle or $6$-cycle exists in the graph of $W$, then $W$ does not resolve $HG(\mathbf{n};3)$, and so $W\cup\{u\}$ does not resolve $HG(\mathbf{n}+\mathbf{1};3)$. This is because the new landmark $u$ is adjacent to all of the vertices in $V_n$, so vertices in $V_n$ can only be resolved using a landmark from $W$.
    If a rainbow triangle exists in the graph of $W$, then a rainbow triangle and triple loop exist in the graph of $W\cup\{u\}$, so by \cref{ex:tripleloop}, $W\cup\{u\}$ does not resolve $HG(\mathbf{n}+\mathbf{1};3)$.
    
    Conversely, suppose $W$ is a $2$-basic landmark system such that $\mathcal{G}(W)$ avoids forbidden $4$-cycles, forbidden $6$-cycles, and rainbow triangles. We will show that $W \cup \{u\}$ is a resolving set for $HG(\mathbf{n}+\mathbf{1};3)$.  
    
    Suppose $\alpha=(a_1, a_2, a_3)$ and $\beta=(b_1, b_2, b_3)$ are distinct vertices in $HG(\mathbf{n}+\mathbf{1};3)$ and that neither is a landmark. 
    Note that, by \cref{lem:basicverboten}, $W$ is a resolving set for $HG(\mathbf{n};3)$, so if $\alpha$ and $\beta$ are both in $V_n$, then they are resolved by some landmark in $W$.  If one vertex, say $\alpha$, is in $V_n$ and the other vertex, $\beta$, is in $V_{n+1}-V_n$, then $u$ resolves $\alpha$ and $\beta$, as $u$ sees $\beta$ but not $\alpha$.  
    The remaining case, where $\alpha$ and $\beta$ are both in $V_{n+1}-V_n$, requires the most work. We will show the edges $W_{i,a_i}$ cannot cover the same set of vertices as the edges $W_{i,b_i}$. We consider cases based on the number of vertices covered. 

    Note that the edges $W_{i,a_i}$ (likewise, $W_{i,b_i}$) must consist of either a loop and two plain edges, or a double loop and a plain edge. Thus the number of vertices covered ranges between three and five.
    
    Case 1 ($3$ vertices). If the edges $W_{i,a_i}$ and the edges $W_{i,b_i}$ cover the same set of three vertices, then each set of edges induces a rainbow $P_2 \cup L_2$, where the $L_2$ must be at the same vertex. And then we must have $W_{i,a_i}=W_{i,b_i}$ for $i=1,2,3$ since there are no double edges. But this contradicts our assumption that $\alpha$ and $\beta$ are distinct.
  
    Case 2 ($4$ vertices). Suppose the edges $W_{i,a_i}$ and the edges $W_{i,b_i}$ cover the same set of four vertices.  Both edge sets induce a rainbow  $P_3 \cup L_1$.  Since the loops occur at the same vertex, the only way for two different rainbow $P_3 \cup L_1$ subgraphs to cover all four vertices is if they together make a $C_3\cup L_2$, but we are assuming that $\mathcal{G}(W)$ avoids rainbow triangles.
    
    Case 3 ($5$ vertices). Suppose the edges $W_{i,a_i}$ and the edges $W_{i,b_i}$ cover the same set of five vertices.
    Both edge sets induce a rainbow  $P_2 \cup P_2 \cup L_1$. 
   Both sets must have their $L_1$ at the same vertex (although possibly different colors).   
    There is no way to make two different rainbow $2P_2$'s cover the same set of $4$ vertices since once the first rainbow $2P_2$ is placed, the only other possible $2P_2$ would have edges of the same color instead of different colors. Hence we must have $W_{i,a_i}=W_{i,b_i}$ for $i=1,2,3$, which contradicts $\alpha$ and $\beta$ being distinct.
    
    Thus in all cases, the edges $W_{i,a_i}$ and the edges $W_{i,b_i}$ cannot cover the same set of vertices and so $\alpha$ and $\beta$ are resolved.  Hence $W \cup \{u\}$ is a resolving set.
\end{proof}

%%%%%%%%%%%%%%%%%%%%%%%%%%%%%%%%%%%%%%%%%%%%%%%%%%%%%%%%%%%%%%%%%%%%
%%%%%%%%%%%%%%%%%%%%%%%%%%%%%%%%%%%%%%%%%%%%%%%%%%%%%%%%%%%%%%%%%%%%
\section{Metric Bases for the Diagonal Family}\label{sec:diagonal}

Now that we know which configurations must be avoided in the landmark graph, we are able to construct minimum resolving sets for the diagonal family of Hamming graphs $HG(n,n,n;3)$. The graphs with $n\geq 5$ achieve the metric dimension lower bound of $2n-1$, while the graphs for $n=3$ and $n=4$ require $2n$ landmarks.

%%%%%%%%%%%%%%%%%%%%%%%%%%%%%%%%%%%%%%%%%%%%%%%%%%%%%%%%%%%%%%%%%%
\subsection{Construction of metric bases}
For the main case of $n\geq 5$, our strategy will be to construct an order $2(n-1)$ graph $G$ that avoids the configurations from \cref{lem:verboten} and define a $2$-basic landmark system $W$ with $\mathcal{G}(W)=G$. We then extend $W$ to a triple-looped landmark system that resolves $HG(n, n, n; 3)$ in $2n-1$ landmarks.

%%%%%%%%%%%%%%%%%%%%%%%%%%

\begin{lemma}\label{lem:cubic}
    For $k\geq 6$, there exists a cubic graph on $2k$ vertices with a proper $3$-edge-coloring that avoids the forbidden edge-colored subgraphs from \cref{lem:verboten}: \begin{enumerate}
    \item a $4$-cycle that contains all three colors of edges,
    \item a $6$-cycle with opposite edges of the same color, and
    \item a rainbow triangle.
\end{enumerate}
\end{lemma}

\begin{proof}
    To construct a graph with the desired properties (see \cref{fig:cubic}), we let $k\geq 6$ and start with an order $2k$ Hamiltonian cycle whose edge colors alternate, say blue and pink. Next, we label the vertices of the cycle in a way that helps us define green edges and avoid the forbidden configurations. Let $m=k-4$.
    If $k$ is even, label the vertices of the cycle in the order
    \[(p_1,p_2,p_3,p_4,q_1,\ldots,q_m,p_1',p_3',p_2',p_4',q_m',\ldots,q_1'),\]
    and if $k$ is odd, label the vertices in the order
    \[(p_1,p_2,p_3,p_4,q_1,\ldots,q_m,p_2',p_4',p_1',p_3',q_m',\ldots,q_1').\]
    To finish the graph, add green edges $p_ip_i'$ (for $i=1,2,3,4$) and $q_iq_i'$ (for $i=1,\ldots,m$). We claim these graphs avoid the forbidden triangles, $4$-cycles, and $6$-cycles of \cref{lem:verboten}.
    
    If there was a rainbow triangle, the outer Hamiltonian cycle would contain a sequence of vertices $u,v,u'$, which we don't have.
    
    If there was a forbidden $4$-cycle with two blue edges or two pink edges, then we would have a path $uvwu'$ along the outer Hamiltonian cycle, but this does not occur ($u$ and $u'$ are always at least five edges apart). If there was a forbidden $4$-cycle with two green edges, then there would exist edges $uv$ and $u'v'$ of different colors, but this does not occur in our construction.
    
    Lastly, if there was a forbidden $6$-cycle, then it would include a path $uvw$ on the outer Hamiltonian cycle such that $u'$ and $w'$ are distance two apart on the Hamiltonian cycle. This would imply $u$ and $w$ are both $q$'s, but then we see that the coloring of the $6$-cycle is not a forbidden coloring.
\end{proof}
%%%%%%%PUT TIKZ HERE
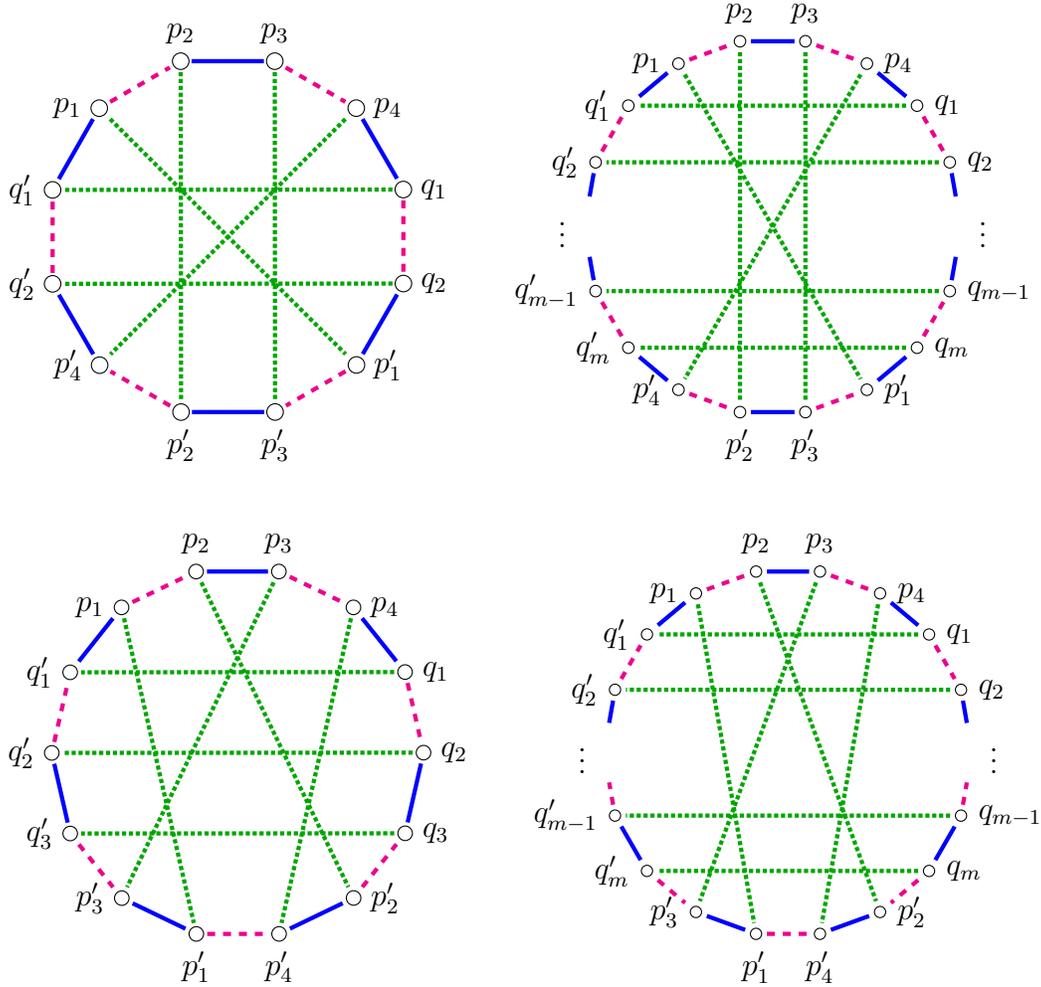
\begin{figure}
    \centering
\text{}\hfill
\begin{tikzpicture}[scale=1.25] % Graph for k=6
              \draw (0,0) node(p_1) {};
              \draw (p_1)+(30:1) node(p_2) {};
              \draw (p_2)+(1,0) node(p_3) {};
              \draw (p_3)+(-30:1) node(p_4) {};
              \draw (p_4)+(-60:1) node(q_1) {};
              \draw (q_1)+(0,-1) node(q_2) {};
              \draw (q_2)+(-120:1) node(p_1') {};
              \draw (p_1')+(-150:1) node(p_3') {};
              \draw (p_3')+(-1,0) node(p_2') {};
              \draw (p_2')+(150:1) node(p_4') {};
              \draw (p_4')+(120:1) node(q_2') {};
              \draw (q_2')+(90:1) node(q_1') {};
              \draw[pinkedge] (p_1)--(p_2);
              \draw[pinkedge] (p_3)--(p_4);
              \draw[pinkedge] (q_1)--(q_2);
              \draw[pinkedge] (p_1')--(p_3');
              \draw[pinkedge] (p_2')--(p_4');
              \draw[pinkedge] (q_2')--(q_1');
              \draw[blueedge] (p_2)--(p_3);
              \draw[blueedge] (p_4)--(q_1);
              \draw[blueedge] (q_2)--(p_1');
              \draw[blueedge] (p_3')--(p_2');
              \draw[blueedge] (p_4')--(q_2');
              \draw[blueedge] (q_1')--(p_1);
              \draw[greenedge] (p_1)--(p_1');
              \draw[greenedge] (p_2)--(p_2');
              \draw[greenedge] (p_3)--(p_3');
              \draw[greenedge] (p_4)--(p_4');
              \draw[greenedge] (q_1)--(q_1');
              \draw[greenedge] (q_2)--(q_2');
              \foreach \pt/\labpos in {p_1/left,p_2/above,p_3/above,p_4/right,q_1/right,q_2/right,p_1'/right,p_3'/below,p_2'/below,p_4'/left,q_2'/left,q_1'/left}
                 \draw[fill=white] (\pt) circle (2.5pt) node[\labpos=2.5pt] {$\pt$};
    \end{tikzpicture}\hfill
       \begin{tikzpicture}[scale=0.87] % Graph for general k even
              \draw (0,0) node(p_1) {};
              \draw (p_1)+(360/18:1) node(p_2) {};
              \draw (p_2)+(1,0) node(p_3) {};
              \draw (p_3)+(-360/18:1) node(p_4) {};
              \draw (p_4)+(-2*360/18:1) node(q_1) {};
              \draw (q_1)+(-3*360/18:1) node(q_2) {};
              \draw (q_2)+(-4*360/18:1) node(q_3) {};
              \draw (q_3)+(-5*360/18:1) node(q_4) {};
              \draw (q_4)+(-6*360/18:1) node(q_5) {};
              \draw (q_5)+(-7*360/18:1) node(p_1') {};
              \draw (p_1')+(-8*360/18:1) node(p_3') {};
              \draw (p_3')+(-9*360/18:1) node(p_2') {};
              \draw (p_2')+(-10*360/18:1) node(p_4') {};
              \draw (p_4')+(-11*360/18:1) node(q_5') {};
              \draw (q_5')+(-12*360/18:1) node(q_4') {};
              \draw (q_4')+(-13*360/18:1) node(q_3') {};
              \draw (q_3')+(-14*360/18:1) node(q_2') {};
              \draw (q_2')+(-15*360/18:1) node(q_1') {};
              \draw (q_2)+(-4*360/18:.7) node(mq23) {};
              \draw (q_4)+(4*360/18:.7) node(mq43) {};
              \draw (q_4')+(-13*360/18:.7) node(mq43') {};
              \draw (q_2')+(13*360/18:.7) node(mq23') {};
              \draw[pinkedge] (p_1)--(p_2);
              \draw[pinkedge] (p_3)--(p_4);
              \draw[pinkedge] (q_1)--(q_2);
              \draw[blueedge] (p_2')--(p_3');
               \draw[blueedge] (q_4')--(mq43');
              \draw[pinkedge] (q_2')--(q_1');
               \draw[blueedge] (q_4)--(mq43);
              \draw[blueedge] (q_5')--(p_4');
              \draw[blueedge] (q_5)--(p_1');
              \draw[blueedge] (p_2)--(p_3);
              \draw[blueedge] (p_4)--(q_1);
              \draw[blueedge] (q_2)--(mq23);
              \draw[pinkedge] (p_3')--(p_1');
              \draw[blueedge] (mq23')--(q_2');
              \draw[blueedge] (q_1')--(p_1);
              \draw[pinkedge] (p_4')--(p_2');
              \draw[pinkedge] (q_4')--(q_5');
              \draw[pinkedge] (q_4)--(q_5);
              \draw[greenedge] (p_1)--(p_1');
              \draw[greenedge] (p_2)--(p_2');
              \draw[greenedge] (p_3)--(p_3');
              \draw[greenedge] (p_4)--(p_4');
              \draw[greenedge] (q_1)--(q_1');
              \draw[greenedge] (q_2)--(q_2');
            %   \draw[greenedge] (q_3)--(q_3');
              \draw[greenedge] (q_4)--(q_4');
              \draw[greenedge] (q_5)--(q_5');
              % draw vertices
              \foreach \pt/\labpos in {p_1/left,p_2/above,p_3/above,p_4/right,q_1/right,q_2/right,p_1'/right,p_3'/below,p_2'/below,p_4'/left,q_2'/left,q_1'/left}
                 \draw[fill=white] (\pt) circle (2.5pt) node[\labpos=2.5pt] {$\pt$};
              \draw[fill=white] (q_5) circle (2.5pt) node[right=2.5pt] {$q_m$};
              \draw[fill=white] (q_4) circle (2.5pt) node[right=2.5pt] {$q_{m-1}$};
              \draw (q_3) node[right=2.5pt] {$\vdots$};
              \draw[fill=white] (q_5') circle (2.5pt) node[left=2.5pt] {$q_m'$};
              \draw[fill=white] (q_4') circle (2.5pt) node[left=2.5pt] {$q_{m-1}'$};
              \draw (q_3') node[left=2.5pt] {$\vdots$};
    \end{tikzpicture} \hfill\text{}
\vskip19pt
\hfill
    \begin{tikzpicture}[scale=1.1] % Graph for k=7
              \draw (0,0) node(p_1) {};
              \draw (p_1)+(360/14:1) node(p_2) {};
              \draw (p_2)+(1,0) node(p_3) {};
              \draw (p_3)+(-360/14:1) node(p_4) {};
              \draw (p_4)+(-2*360/14:1) node(q_1) {};
              \draw (q_1)+(-3*360/14:1) node(q_2) {};
              \draw (q_2)+(-4*360/14:1) node(q_3) {};
              \draw (q_3)+(-5*360/14:1) node(p_2') {};
              \draw (p_2')+(-6*360/14:1) node(p_4') {};
              \draw (p_4')+(-7*360/14:1) node(p_1') {};
              \draw (p_1')+(-8*360/14:1) node(p_3') {};
              \draw (p_3')+(-9*360/14:1) node(q_3') {};
              \draw (q_3')+(-10*360/14:1) node(q_2') {};
              \draw (q_2')+(-11*360/14:1) node(q_1') {};
              \draw[pinkedge] (p_1)--(p_2);
              \draw[pinkedge] (p_3)--(p_4);
              \draw[pinkedge] (q_1)--(q_2);
              \draw[pinkedge] (p_1')--(p_4');
              \draw[pinkedge] (q_3')--(p_3');
              \draw[pinkedge] (q_2')--(q_1');
              \draw[pinkedge] (p_2')--(q_3);
              \draw[blueedge] (p_2)--(p_3);
              \draw[blueedge] (p_4)--(q_1);
              \draw[blueedge] (q_2)--(q_3);
              \draw[blueedge] (p_3')--(p_1');
              \draw[blueedge] (q_3')--(q_2');
              \draw[blueedge] (q_1')--(p_1);
              \draw[blueedge] (p_4')--(p_2');
              \draw[greenedge] (p_1)--(p_1');
              \draw[greenedge] (p_2)--(p_2');
              \draw[greenedge] (p_3)--(p_3');
              \draw[greenedge] (p_4)--(p_4');
              \draw[greenedge] (q_1)--(q_1');
              \draw[greenedge] (q_2)--(q_2');
              \draw[greenedge] (q_3)--(q_3');
              \foreach \pt/\labpos in {p_1/left,p_2/above,p_3/above,p_4/right,q_1/right,q_2/right,p_1'/below,p_3'/left,p_2'/right,p_4'/below,q_2'/left,q_1'/left,q_3/right,q_3'/left}
                 \draw[fill=white] (\pt) circle (2.5pt) node[\labpos=2.5pt] {$\pt$};
    \end{tikzpicture}\hfill
    \begin{tikzpicture}[scale=.85] % Graph for general k odd
              \draw (0,0) node(p_1) {};
              \draw (p_1)+(360/18:1) node(p_2) {};
              \draw (p_2)+(1,0) node(p_3) {};
              \draw (p_3)+(-360/18:1) node(p_4) {};
              \draw (p_4)+(-2*360/18:1) node(q_1) {};
              \draw (q_1)+(-3*360/18:1) node(q_2) {};
              \draw (q_2)+(-4*360/18:1) node(q_3) {};
              \draw (q_3)+(-5*360/18:1) node(q_4) {};
              \draw (q_4)+(-6*360/18:1) node(q_5) {};
              \draw (q_5)+(-7*360/18:1) node(p_2') {};
              \draw (p_2')+(-8*360/18:1) node(p_4') {};
              \draw (p_4')+(-9*360/18:1) node(p_1') {};
              \draw (p_1')+(-10*360/18:1) node(p_3') {};
              \draw (p_3')+(-11*360/18:1) node(q_5') {};
              \draw (q_5')+(-12*360/18:1) node(q_4') {};
              \draw (q_4')+(-13*360/18:1) node(q_3') {};
              \draw (q_3')+(-14*360/18:1) node(q_2') {};
              \draw (q_2')+(-15*360/18:1) node(q_1') {};
              \draw (q_2)+(-4*360/18:.7) node(mq23) {};
              \draw (q_4)+(4*360/18:.7) node(mq43) {};
              \draw (q_4')+(-13*360/18:.7) node(mq43') {};
              \draw (q_2')+(13*360/18:.7) node(mq23') {};
              \draw[pinkedge] (p_1)--(p_2);
              \draw[pinkedge] (p_3)--(p_4);
              \draw[pinkedge] (q_1)--(q_2);
              \draw[pinkedge] (p_1')--(p_4');
               \draw[pinkedge] (q_4')--(mq43');
              \draw[pinkedge] (q_2')--(q_1');
               \draw[pinkedge] (q_4)--(mq43);
              \draw[pinkedge] (q_5')--(p_3');
              \draw[pinkedge] (q_5)--(p_2');
              \draw[blueedge] (p_2)--(p_3);
              \draw[blueedge] (p_4)--(q_1);
              \draw[blueedge] (q_2)--(mq23);
              \draw[blueedge] (p_3')--(p_1');
              \draw[blueedge] (mq23')--(q_2');
              \draw[blueedge] (q_1')--(p_1);
              \draw[blueedge] (p_4')--(p_2');
              \draw[blueedge] (q_4')--(q_5');
              \draw[blueedge] (q_4)--(q_5);
              \draw[greenedge] (p_1)--(p_1');
              \draw[greenedge] (p_2)--(p_2');
              \draw[greenedge] (p_3)--(p_3');
              \draw[greenedge] (p_4)--(p_4');
              \draw[greenedge] (q_1)--(q_1');
              \draw[greenedge] (q_2)--(q_2');
            %   \draw[greenedge] (q_3)--(q_3');
              \draw[greenedge] (q_4)--(q_4');
              \draw[greenedge] (q_5)--(q_5');
              % draw vertices
              \foreach \pt/\labpos in {p_1/left,p_2/above,p_3/above,p_4/right,q_1/right,q_2/right,p_1'/below,p_3'/left,p_2'/right,p_4'/below,q_2'/left,q_1'/left}
                 \draw[fill=white] (\pt) circle (2.5pt) node[\labpos=2.5pt] {$\pt$};
              \draw[fill=white] (q_5) circle (2.5pt) node[right=2.5pt] {$q_m$};
              \draw[fill=white] (q_4) circle (2.5pt) node[right=2.5pt] {$q_{m-1}$};
              \draw (q_3) node[right=2.5pt] {$\vdots$};
              \draw[fill=white] (q_5') circle (2.5pt) node[left=2.5pt] {$q_m'$};
              \draw[fill=white] (q_4') circle (2.5pt) node[left=2.5pt] {$q_{m-1}'$};
              \draw (q_3') node[left=2.5pt] {$\vdots$};
    \end{tikzpicture}\hfill\text{}
    \caption{Cubic graph on $2k$ vertices with a proper $3$-edge-coloring for even $k\geq 6$ (top row images) and odd $k\geq7$ (bottow row images). Edges of outer Hamiltonian cycle alternate blue (solid) and pink (dashed). All other edges are green (dotted). Forbidden $4$-cycles, forbidden $6$-cycles, and rainbow triangles are avoided.}
    \label{fig:cubic}
\end{figure}
\begin{remark}\label{nogoodcubic10}
Note that \cref{lem:cubic} is not true for $k=5$.
Using \texttt{GraphData} in Mathematica \cite{Mathematica}, we find that, of the $21$ cubic graphs with $10$ vertices ($19$ of which are connected \cite{Bussemaker1977}), there are only five that are triangle-free and have edge-chromatic number $3$. Moreover, it can be verified that all possible proper edge-colorings of these graphs lead to a forbidden configuration.
\end{remark}

In \cref{mdim:ngeq7}, we show that for $n\geq 5$, the diagonal family $HG(n,n,n;3)$ achieves the metric dimension lower bound $2n-1$. To prove the cases $n\geq 7$, we make use of the graphs from \cref{lem:cubic}, and to prove the cases $n=5$ and $n=6$ we use ad hoc constructions.

\begin{proposition}\label{mdim:ngeq7}
  For $n\geq 5$, the Hamming graph $HG(n,n,n;3)$ has metric dimension $2n-1$.
\end{proposition}

\begin{proof}
    For $n\geq 7$, let $k=n-1$ and let $G$ be the graph constructed as in \cref{lem:cubic}.
    Label the blue edges $1,\ldots,n-1$ and likewise for the green and pink edges. If the blue, green, and pink edges incident to a vertex are labeled $a_1$, $a_2$, and $a_3$, respectively, then we create a landmark $(a_1,a_2,a_3)$. Do this for each vertex in $G$ to get a landmark set $W$ with graph $\mathcal{G}(W)=G$. The order $2(n-1)$ graph $G$ avoids forbidden $4$-cycles and $6$-cycles and rainbow triangles, so by \cref{lem:verboten}, if we let $u=(n,n,n)$, then the set $W\cup\{u\}$ is a resolving set of size $2n-1$ for $HG(n,n,n;3)$.
    By the lower bound in \cref{cor:lowerbound}, $W\cup\{u\}$ is a minimum resolving set.
    
    For $n=5$, note that \cref{lem:cubic} is also true for $k=4$. As illustrated in \cref{fig:fourfive}, there is a proper $3$-edge coloring of the order $8$ M\"{o}bius ladder that avoids rainbow triangles, forbidden $4$-cycles, and forbidden $6$-cycles. Thus we can resolve $HG(5,5,5;3)$ in the style of the previous paragraph, taking $G$ to be the colored M\"{o}bius ladder.
    
    For $n=6$, we take a different approach, as \cref{lem:cubic} is not true for $k=5$ (see \cref{nogoodcubic10}). In this case, rather than concentrating all loops at the same vertex, we spread them throughout the landmark graph. Letting symbol $k$ in row $i$, column $j$ correspond to a landmark $(i,j,k)$, we can verify by computer that the partial Latin square
    \[\begin{array}{|c|c|c|c|c|c|}
     \hline
      & 1& & & &  \\ \hline
      1 & 2 & & & &  \\ \hline
       & & 2& 3 & &  \\ \hline
       & & 3 &4 & &  \\ \hline
       &  & &  & 4& 5 \\ \hline
       & &  & & 5 & 6 \\ \hline
  \end{array}\]
     represents a resolving set of size $11$ for $HG(6,6,6;3)$.\footnote{Incidentally, our approach to constructing this resolving set of size $11$ was to attempt to prove there was none. We failed and found this one! Our code for checking resolving sets with SageMath \cite{SageMath} is available at \url{https://github.com/fostergreenwood/metric-dimension}.}
\end{proof}

%%%%%%%%%%%%%%%%%%%%%%%%%%%%%%%%%%%%%%%%%%%%%%%%%%%%%
\begin{figure}
    \centering
    \begin{tikzpicture}[scale=1.75] % Moebius ladder
        \draw[blueedge] (337.5:.75)--(22.5:.75);
        \draw[blueedge] (67.5:.75)--(112.5:.75);
        \draw[blueedge] (157.5:.75)--(202.5:.75);
        \draw[blueedge] (247.5:.75)--(292.5:.75);
        \draw[pinkedge] (22.5:.75)--(67.5:.75);
        \draw[pinkedge] (112.5:.75)--(157.5:.75);
        \draw[pinkedge] (202.5:.75)--(247.5:.75);
        \draw[pinkedge] (292.5:.75)--(337.5:.75);
        \draw[greenedge] (67.5:.75)--(247.5:.75);
        \draw[greenedge] (22.5:.75)--(202.5:.75);
        \draw[greenedge] (337.5:.75)--(157.5:.75);
        \draw[greenedge] (112.5:.75)--(292.5:.75);
        \foreach \x in {1,...,8}
            \draw[fill=white] ({22.5+45*\x}:.75) circle (2pt);
    \end{tikzpicture}
    \caption{Proper $3$-edge coloring of the order eight M\"{o}bius ladder $M_4$. Edges of outer Hamiltonian cycle alternate blue (solid) and pink (dashed), while green (dotted) edges join antipodal points on the cycle. All forbidden configurations are avoided.}
    \label{fig:fourfive}
\end{figure}
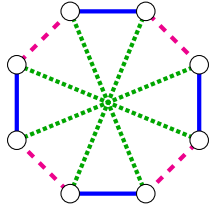

%%%%%%%%%%%%%%%%%%%%%%%%%%%%%%%%%%%%%%%%%%%%%%%%%%%%%%%%%%%%
\subsection{Small exceptions}
In contrast to the diagonal cases $n\geq 5$, which can be resolved in $2n-1$ landmarks, we will show that for $n=3$ and $n=4$, the metric dimension of $HG(n,n,n;3)$ is $2n$. 

To show that $HG(n,n,n;3)$ cannot be resolved in $2n-1$ landmarks for $n=3,4$, we need to consider all possible subsets of $2n-1$ vertices, not only triple-looped landmark systems. With fewer constraints on $W$, there are more possible footprints than shown in \cref{fig:dinosaurtracks}---in general, a footprint may have double edges, and loops do not all have to be at the same vertex. In turn, there are more ways for two different footprints to cover the same set of vertices in the landmark graph, and such configurations indicate pairs of unresolved vertices in the Hamming graph. 

In addition to the forbidden $4$-cycles and $6$-cycles in \cref{fig:verboten}, the landmark graph of a resolving set has to avoid the configurations in \cref{fig:verbotenSnakes}.
The list is not exhaustive but includes what we need in order to prove \cref{threethreethree} and \cref{fourfourfour}. Of the configurations we require, we only show one representative coloring, with the understanding that the colors may be permuted.  Note that which edges are the same color is significant; for example, in the long $2$-headed snake, it is important that the middle two colors match the loop colors on the end, but in the opposite order. 

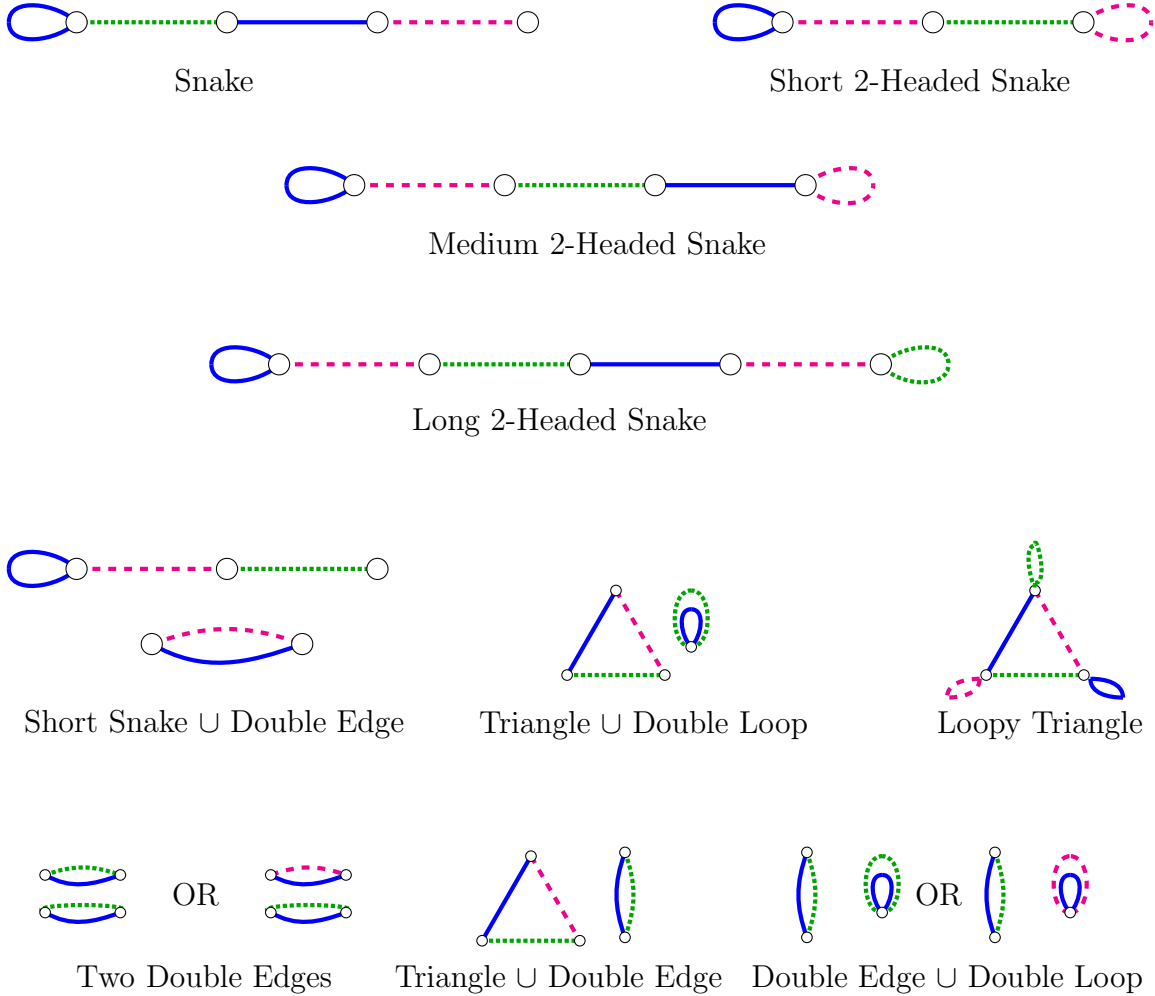
\begin{figure}
    \centering
\begin{tikzpicture}[scale=2] 
% snake
  \draw[greenedge] (0,0)--(1,0);
  \draw[blueedge] (1,0)--(2,0);
  \draw[pinkedge] (2,0)--(3,0);
  \draw[blueedge] (0,0) to[out=225,in=270] (-0.45,0);
  \draw[blueedge] (0,0) to[out=135,in=90] (-0.45,0);
  \draw[fill=white] (0,0) circle (2pt);
  \draw[fill=white] (1,0) circle (2pt);
  \draw[fill=white] (2,0) circle (2pt);
  \draw[fill=white] (3,0) circle (2pt);
  \node at (1.3,0,1) {Snake};
%   \node [below=0.5cm,align=flush center, text width=8cm] at (1,0)
%         {
%             Snake
%         };
\end{tikzpicture}\hfill
\begin{tikzpicture}[scale=2]
% short 2-headed snake
    \draw[greenedge] (1,0)--(2,0);
    \draw[pinkedge] (0,0)--(1,0);
    \draw[blueedge] (0,0) to[out=225,in=270] (-0.45,0);
    \draw[blueedge] (0,0) to[out=135,in=90] (-0.45,0);
    \draw[pinkedge] (2,0) to[out=305,in=270] (2.45,0);
    \draw[pinkedge] (2,0) to[out=45,in=90] (2.45,0);
    \draw[fill=white] (0,0) circle (2pt);
    \draw[fill=white] (1,0) circle (2pt);
    \draw[fill=white] (2,0) circle (2pt);
    \node at (1.3,0,1) {Short 2-Headed Snake};
    % \node [below=.5cm, align=flush center,text width=8cm] at (1,0)
    %     {
    %         Short 2-Headed Snake
    %     };
\end{tikzpicture}\hfill\vskip20pt
\begin{tikzpicture}[scale=2]
% medium 2-headed snake.
    \draw[greenedge] (1,0)--(2,0);
    \draw[pinkedge] (0,0)--(1,0);
    \draw[blueedge] (2,0)--(3,0);
    \draw[blueedge] (0,0) to[out=225,in=270] (-0.45,0);
    \draw[blueedge] (0,0) to[out=135,in=90] (-0.45,0);
    \draw[pinkedge] (3,0) to[out=305,in=270] (3.45,0);
    \draw[pinkedge] (3,0) to[out=45,in=90] (3.45,0);
    \draw[fill=white] (0,0) circle (2pt);
    \draw[fill=white] (1,0) circle (2pt);
    \draw[fill=white] (2,0) circle (2pt);
    \draw[fill=white] (3,0) circle (2pt);
    \node at (2,0,1) {Medium 2-Headed Snake};
    %   \node [below=.5cm, align=flush center,text width=8cm] at (1,0)
    %     {
    %         Medium 2-Headed Snake
    %     };
\end{tikzpicture}\hfill\vskip26pt
\begin{tikzpicture}[scale=2]
%long 2-headed snake.
    \draw[pinkedge] (0,0)--(1,0);
    \draw[greenedge] (1,0)--(2,0);
    \draw[blueedge] (2,0)--(3,0);
    \draw[pinkedge] (3,0)--(4,0);
    \draw[blueedge] (0,0) to[out=225,in=270] (-0.45,0);
    \draw[blueedge] (0,0) to[out=135,in=90] (-0.45,0);
    \draw[greenedge] (4,0) to[out=305,in=270] (4.45,0);
    \draw[greenedge] (4,0) to[out=45,in=90] (4.45,0);
    \draw[fill=white] (0,0) circle (2pt);
    \draw[fill=white] (1,0) circle (2pt);
    \draw[fill=white] (2,0) circle (2pt);
    \draw[fill=white] (3,0) circle (2pt);
    \draw[fill=white] (4,0) circle (2pt);
    \node at (2.25,0,1)  {Long 2-Headed Snake};
    %   \node [below=.5cm,align=flush center,text width=8cm] at (1,0)
    %     {
    %         Long 2-Headed Snake
    %     };
\end{tikzpicture}\hfill\vskip30pt
\begin{tikzpicture}[scale=2]
%poor snake
    % short snake
    \draw[pinkedge] (0,0)--(1,0);
    \draw[greenedge] (1,0)--(2,0);
    \draw[blueedge] (0,0) to[out=225,in=270] (-0.45,0);
    \draw[blueedge] (0,0) to[out=135,in=90] (-0.45,0);
    % double edge
    \draw[blueedge] (0.5,-0.5) to[out=330,in=200] (1.5,-0.5);
    \draw[pinkedge] (0.5,-0.5) to[out=20,in=160] (1.5,-0.5);  
    % vertices
    \draw[fill=white] (0,0) circle (2pt);
    \draw[fill=white] (1,0) circle (2pt);
    \draw[fill=white] (2,0) circle (2pt);
    \draw[fill=white] (0.5,-0.5) circle (2pt);
    \draw[fill=white] (1.5,-0.5) circle (2pt);
    \node at (1.3,-.65,1) {Short Snake $\cup$ Double Edge};
  %    \node [below=1.5cm, align=flush center,text width=8cm] at (1,0)
   %     {
    %        Poor Snake
     %   };
\end{tikzpicture}\hfill
\begin{tikzpicture}
 % triangle + double loop
    % triangle
    \draw[blueedge] ({330+120}:.75)--({330+240}:.75);
    \draw[greenedge] ({330+240}:.75)--({330}:.75);
    \draw[pinkedge] ({330+120}:.75)--({330}:.75);
    \foreach \x in {1,2,3}
        \draw[fill=white] ({330+120*\x}:.75) circle (2pt);
    % double loop
    \draw[blueedge] (1,0) to[out=135,in=180] (1,0.5);
    \draw[blueedge] (1,0) to[out=45,in=0] (1,0.5);
    \draw[greenedge] (1,0) to[out=180,in=180] (1,0.75);
    \draw[greenedge] (1,0) to[out=0,in=0] (1,0.75);
    \draw[fill=white] (1,0) circle (2pt);
    \node at (0.75,-.65,1) {Triangle $\cup$ Double Loop};
       %\node [below=0.5cm, align=flush center,text width=8cm] at (1,0)
    %    {
     %       Triangle + loop
      %  };
\end{tikzpicture}\hfill\hfill
\begin{tikzpicture}
%triangle with loop on each corner
    % triangle
    \draw[blueedge] ({330+120}:.75)--({330+240}:.75);
    \draw[greenedge] ({330+240}:.75)--({330}:.75);
    \draw[pinkedge] ({330+120}:.75)--({330}:.75);
    \foreach \x in {1,2,3}
        \draw[fill=white] ({330+120*\x}:.75) circle (2pt);
    % corner loops
    \draw[blueedge] ({330}:0.85) to[out=0,in=90] ({330}:1.35);
    \draw[blueedge] ({330}:.85) to[out=270,in=180] ({330}:1.35);
    \draw[greenedge] (450:.85) to[out=270,in=180] (450:1.35);
    \draw[greenedge] (450:.85) to[out=0,in=90] (450:1.35);
    \draw[pinkedge] (570:.85) to[out=180,in=90] (570:1.35);
    \draw[pinkedge] (570:.85) to[out=270,in=0] (570:1.35);
    \node at (0.45,-.65,1) {Loopy Triangle};
        %      \node [below=0.6cm, align=flush center,text width=8cm] at (1,0)
        % {
        %     Loopy Triangle
        % };
\end{tikzpicture}\hfill\vskip36pt\hfill
\begin{tikzpicture}
% two double edges
    % first version (use only two colors)
    \draw[blueedge] (-3,0)to[out=330,in=200](-2,0);
    \draw[greenedge] (-2,0)to[out=20,in=160](-3,0);
    
    \draw[blueedge] (-3,.5) to[out=330,in=200] (-2,.5);
    \draw[greenedge] (-3,.5) to[out=20,in=160] (-2,.5);
  
    \draw[fill=white] (-3,0) circle (2pt);
    \draw[fill=white] (-2,0) circle (2pt);
    \draw[fill=white] (-3,.5) circle (2pt);
    \draw[fill=white] (-2,.5) circle (2pt);
  %%%%%%%%%%%%%
  \node at (-1,0.25){OR};
  %%%%%%%%%%%%%
    % second version (use all three colors)
    \draw[blueedge] (0,0)to[out=330,in=200](1,0);
    \draw[greenedge] (1,0)to[out=20,in=160](0,0);
    
    \draw[blueedge] (0,.5) to[out=330,in=200] (1,.5);
    \draw[pinkedge] (0,.5) to[out=20,in=160] (1,.5);
  
    \draw[fill=white] (0,0) circle (2pt);
    \draw[fill=white] (1,0) circle (2pt);
    \draw[fill=white] (0,.5) circle (2pt);
    \draw[fill=white] (1,.5) circle (2pt);
    \node at (-.5,-.5,1) {Two Double Edges};
        %  \node [below=0.5cm, align=flush center,text width=8cm] at (1,0)
        % {
        %     Two Double Edges
        % };
\end{tikzpicture}\hfill
\begin{tikzpicture}
% Triangle U Double Edge
    % triangle
    \draw[blueedge] ({330+120}:.75)--({330+240}:.75);
    \draw[greenedge] ({330+240}:.75)--({330}:.75);
    \draw[pinkedge] ({330+120}:.75)--({330}:.75);
    \foreach \x in {1,2,3}
        \draw[fill=white] ({330+120*\x}:.75) circle (2pt);
    % double edge
    \draw[blueedge] (1.25,.8) to[out=250,in=110] (1.25,-0.33);
    \draw[greenedge] (1.25,.8) to[out=290,in=70] (1.25,-0.33);
    \draw[fill=white] (1.25,0.8) circle (2pt);
    \draw[fill=white] (1.25,-0.33) circle (2pt);
    \node at (0.75,-.5,1) {Triangle $\cup$ Double Edge};
        %       \node [below=0.5cm,align=flush center,text width=8cm] at (1,0)
        % {
        %     Triangle + Double Edge
        % };
\end{tikzpicture}\hfill
\begin{tikzpicture}
% D_2 \cup L_2
    % first version (uses only two colors)
    % double edge
    \draw[blueedge] (-1.5,.8) to[out=250,in=110] (-1.5,-0.33);
    \draw[greenedge] (-1.5,.8) to[out=290,in=70] (-1.5,-0.33);   
    \draw[fill=white] (-1.5,0.8) circle (2pt);
    \draw[fill=white] (-1.5,-0.33) circle (2pt);
    % double loop
    \draw[blueedge] (-.5,0) to[out=135,in=180] (-.5,0.5);
    \draw[blueedge] (-.5,0) to[out=45,in=0] (-.5,0.5);
    \draw[greenedge] (-.5,0) to[out=180,in=180] (-.5,0.75);
    \draw[greenedge] (-.5,0) to[out=0,in=0] (-.5,0.75);    
    \draw[fill=white] (-0.5,0) circle (2pt);
        %%%%%
        \node at (0.25,0.25){OR};
        %%%%%
    % second version (uses all three colors)
    % double edge        
    \draw[blueedge] (1,.8) to[out=250,in=110] (1,-0.33);
    \draw[greenedge] (1,.8) to[out=290,in=70] (1,-0.33);
    \draw[fill=white] (1,0.8) circle (2pt);
    \draw[fill=white] (1,-0.33) circle (2pt);
    % double loop
    \draw[blueedge] (2,0) to[out=135,in=180] (2,0.5);
    \draw[blueedge] (2,0) to[out=45,in=0] (2,0.5);
    \draw[pinkedge] (2,0) to[out=180,in=180] (2,0.75);
    \draw[pinkedge] (2,0) to[out=0,in=0] (2,0.75);
    \draw[fill=white] (2,0) circle (2pt);
    \node at (0.75,-.5,1) {Double Edge $\cup$ Double Loop};
\end{tikzpicture}\hfill
    \caption{Some configurations of two different footprints covering the same set of vertices. Different line styles correspond to different edge colors. Configurations with only two colors imply the addition of any edge of the third color will cause an issue.}
    \label{fig:verbotenSnakes}
\end{figure}

\begin{lemma}\label{threethreethree}
    The Hamming graph $HG(3,3,3;3)$ has metric dimension $6$.
\end{lemma}

\begin{proof}
Letting symbol $k$ in row $i$, column $j$ corresponds to a landmark $(i,j,k)$, we can verify by computer that the partial Latin square
\[\begin{array}{|c|c|c|}
  \hline
     1  & 2 & 3\\ \hline
     3  & 1 & 2\\ \hline
    \phantom{2}   & \phantom{1}& \phantom{3}\\ \hline
  \end{array}\]
represents a resolving set of size $6$ (whose landmark graph is in \cref{fig:landmarkgraph}).
  
To show there is no smaller resolving set, suppose, towards a contradiction, that we have a resolving set $W$ with only $5$ vertices. Then by \cref{howmany}, the graph of $W$ has 2 plain edges and a loop for each of three colors (say blue, green, pink). We will show that all possible placements of the edges and loops lead to a forbidden configuration.

Case 1 (no double edges). First suppose the graph of $W$ has no double edges.

 Case 1.1 (blue and green loop on same vertex).  If the blue and green loops are on the same vertex, say $u$, then the remaining two blue edges and two green edges are forced to create an alternating $4$-cycle. Note that distinct vertices in the $4$-cycle cannot be joined by a pink edge since that would create either a double edge or a forbidden $C_3\cup L_2$. Thus every pink plain edge must have $u$ as one of its endpoints. But this means we can only have one pink plain edge instead of the two required.
 
 Case 1.2 (blue and green loop on different vertices).  Suppose the graph of $W$ has a blue loop at vertex $b$ and a green loop at vertex $g$. Now there must be a green edge, say $bu$, and a blue edge, say $gv$. Note that if we had $u=v$, then the remaining blue and green edges would create a double edge. Thus we have $u\neq v$. Now, if the remaining vertex is labeled $w$, we are forced to have a blue edge $wu$ and a green edge $wv$. We claim there is no way to place a pink at $w$.  A loop at $w$ gives  short two-headed snakes ($buw$ and $gvw$). A plain pink edge from $w$ causes a double edge ($uw$ or $vw$) or a snake ($buwg$ or $gvwb$).
 
 Case 2 (double edge exists). Next, suppose the graph of $W$ has a double edge. Without loss of generality, say it is a blue-green double edge between vertices $u$ and $v$. Then, to avoid a forbidden configuration, the blue and green loops must be on different vertices, say $b$ and $g$. If we label the final vertex as $w$, the only way to place the remaining blue and green edges is with a green edge $bw$ and a blue edge $gw$. We claim there is no way to place a pink at $b$. A plain pink edge 
 $bu$ or $bv$ give us a snake with head $g$. A pink edge $bw$ creates a second double edge. A pink loop at $b$ creates a $D_2\cup L_2$. A pink edge $bg$ creates a forbidden $D_2\cup C_3$. 
 
 Since all possibilities lead to forbidden configurations, there is no way to resolve the Hamming graph $HG(3,3,3;3)$ using only five landmarks.
\end{proof}

\begin{lemma}\label{fourfourfour}
  The Hamming graph $HG(4,4,4;3)$ has metric dimension $8$.
\end{lemma}
\begin{proof} 
To construct a resolving set, let $M_4$ be the order $8$ edge-colored M\"{o}bius ladder in \cref{fig:fourfive}.  Number the blue (likewise, green and pink) edges $1,\ldots,4$.
If the blue, green, and pink edges incident to a vertex are labeled $a_1$, $a_2$, and $a_3$, respectively, then we create a landmark $(a_1,a_2,a_3)$. Do this for each vertex in $M_4$ to get a landmark set $W$ with graph $\mathcal{G}(W)=M_4$. By \cref{lem:basicverboten}, the set $W$ is a resolving set of size $8$ for $HG(4,4,4;3)$.
% We can verify by computer that the following table represents a resolving set of size $8$ by identifying symbol $k$ in row $i$, column $j$ with a landmark $(i,j,k)$.
%  \[\begin{array}{|c|c|c|c|}
%   \hline
%      1  &  &  & 4\\ \hline
%       & 2 & 1 & \\ \hline
%      3 &   & & 2\\ \hline
%      & 4 & 3 & \\ \hline
%   \end{array}\]
  
  To show there is no smaller resolving set, suppose, towards a contradiction, that we have a resolving set $W$ with only $7$ vertices. Then by \cref{howmany}, the graph of $W$ has for each color three plain edges and a loop.
   
   Case 1 (no double edges). First suppose the graph of $W$ has no double edges.
   
   Case 1.1 (blue and green loop on same vertex). Suppose the blue and green loops are on the same vertex, say $u$. Covering the six remaining vertices, we have three blue edges that make a perfect matching and three green edges that make a perfect matching. Since we are in the case of no double edges, it is straightforward to verify that the edges of the two matchings must combine to create an alternating blue-green $6$-cycle.
   
   Now we try to determine the pink edges. Note that distinct vertices in the $6$-cycle cannot be joined by a pink edge since that would create either a double edge, a $C_3\cup L_2$, or a forbidden $4$-cycle. Thus every pink plain edge must have $u$ as one of its endpoints. But this means we can only have one pink plain edge instead of the three required.
   
   Case 1.2 (blue and green loops on different vertices). Suppose the graph of $W$ has a blue loop at vertex $b$ and a green loop at vertex $g$. Now there must be a green edge, say $bu$, and a blue edge, say $gv$.
   
   Case 1.2.1 ($u=v$). If $u=v$, then the remaining blue and green edges create an alternating $4$-cycle. Note that we cannot have a pink edge from $b$ or $g$ to a vertex on the $4$-cycle, since that would create a snake with head $g$ or $b$. We also cannot have a pink loop at $b$ or $g$ since that creates a short two-headed snake $bug$. To avoid double edges, the only remaining option is a pink edge $bg$. Now, we cannot have a pink loop at $u$ (since it would create a triangle with a loop at each vertex), so there must be a pink edge from $u$ to a vertex, say $w$, on the blue-green $4$-cycle. Since we can't have a double edge, the pink loop is forced to be at the vertex $w'$ on the $4$-cycle that is nonadjacent to $w$. But now there is a long two-headed snake (with heads $b$ and $w'$).
   
   Case 1.2.2 ($u\neq v$). If $u\neq v$, then we have a blue edge, say $uu'$ and a green edge $vv'$. Note that $u'\neq v'$ since $u'=v'$ would force a double edge with the remaining blue and green edges. There is only one vertex remaining, call it $w$, and we are forced to have a green edge $wu'$ and a blue edge $wv'$.
   
   Now we consider pink edges at $u'$ and $v'$. Due to the path $buu'$, we cannot have a pink loop at $u'$ (would create a short two-headed snake), and the only way to avoid a double edge or a snake with head $b$ is to put a pink edge $u'b$. Similarly, we must have a pink edge $v'g$. But now we have a long two-headed snake $bu'wv'g$.
   
   Case 2 (there exists a double edge). By symmetry of permuting colors, it suffices to consider the case of a blue-green double edge, say between vertices $u$ and $v$. Then the blue and green loops must be at distinct vertices, say $b$ and $g$ (otherwise we obtain $D_2\cup L_2$). Now, there is a green edge from $b$ to some vertex $u'$ and a blue edge from $g$ to some vertex $v'$. If $u'=v'$, we would be forced to make a blue-green double edge on the remaining two vertices which gives two double edges which is forbidden, so it must be that $u'\neq v'$. Similar to Case 1.2.2, there is only one vertex remaining, call it $w$, and we are forced to create a blue edge $u'w$ and a green edge $v'w$.
   
   Now we consider pink edges. Note that at $w$, we cannot have a pink loop since we would get a short two-headed snake $bu'w$. We also cannot have a pink edge $wu$ or $wv$ since that would create a snake with head $b$ or $g$.  A pink edge $wu'$ or $wv'$ would create two double edges.
   Lastly, a pink edge $wb$ or $wg$ would create a $C_3\cup D_2$. Thus there is no way to place a pink at $w$. This completes the case where the graph of $W$ has a double edge.
   
   Since all possibilities lead to forbidden configurations, there is no way to resolve the Hamming graph $HG(4,4,4;3)$ using only seven landmarks.
\end{proof}

%%%%%%%%%%%%%%%%%%%%%%%%%%%%%%%%%%%%%%%%%%%%%%%%%%%%%%%%%%%%%%%%%%
\subsection{Conclusion and further directions}

Compiling the results of the previous two subsections, we now have the metric dimension for all graphs in the diagonal family $HG(n,n,n;3)$ for $n\geq 3$.

\begin{theorem}\label{thm:ComprehensiveMetDim}
For $n\geq 3$, the metric dimension of the Hamming graph $HG(n,n,n;3)$ is
  \[\dim HG(n,n,n;3)=\begin{cases}
     2n & \text{if $n\in\{3,4\}$} \\
     2n-1 & \text{if $n\geq5$.}
  \end{cases}\]
\end{theorem}

\begin{proof}
    This is a corollary of \cref{mdim:ngeq7}, \cref{threethreethree}, and \cref{fourfourfour}.
\end{proof}

We are currently working to generalize the results of \cref{sec:landmarksystems} and construct minimum resolving sets for non-diagonal graphs $HG(n_1,n_2,n_3;3)$. These graphs are of special interest because when the $n_i$'s are distinct odd primes, the graphs are isomorphic to unitary Cayley graphs \cite{Sander2010}. We know that at least some of these graphs also achieve the lower bound of \cref{cor:lowerbound}, as illustrated in our final example.
  
\begin{example}
 Letting symbol $k$ in row $i$, column $j$ correspond to a landmark $(i,j,k)$, we can check by computer that the partial Latin rectangle
  \[\begin{array}{|c|c|c|c|c|c|c|}
  \hline
     1  & 2 & 3 & 10 & & &  \\
     \hline
     4  & 5 & 6 & 1 & 2 & 3 & \\
     \hline
     7 & 8 & 9 & 4 & 5 & 6 & \\
     \hline
     10 & & & 7 & 8 & 9 & \phantom{10} \\
     \hline
      \phantom{10} &\phantom{10} &\phantom{10} &\phantom{10} & \phantom{10}&\phantom{10} & 11 \\
       \hline
  \end{array}\]
  represents a minimum resolving set of size $21$ for the graph $HG(5,7,11;3)$. 
\end{example}

%%%%%%%%%%%%%%%%%%%%%%%%%%%%%%%%%%%%%%%%%%%%%%%%%%%%%%%%%%%%%%%%%
\bibliographystyle{alphaurl}
\bibliography{metricdimension}

\end{document}